\documentclass[a4paper,11pt,reqno]{amsart}

\usepackage{amsfonts}
\usepackage{amssymb}
\usepackage{amscd}
\usepackage{amsthm}
\usepackage{a4wide}
\usepackage{mathrsfs}
\usepackage[applemac]{inputenc}
\usepackage{amsmath}
\usepackage{amssymb}
\usepackage{amsthm}
\usepackage{textcomp}
\usepackage{graphicx}
\usepackage{enumerate}
\usepackage{mathrsfs}
\usepackage{frcursive}
\usepackage{tikz}
\usepackage[cyr]{aeguill}
\usepackage{xspace}
\usepackage{hyperref}
\usepackage{appendix}

\newtheorem{defn}{Definition}[section]
\newtheorem{lemma}[defn]{Lemma}
\newtheorem{prop}[defn]{Proposition}
\newtheorem{theo}[defn]{Theorem}
\newtheorem{coro}[defn]{Corollary}

\newtheorem{claim}{Claim}
\newtheorem{rk}[defn]{Remark}
\newtheorem{ques}[defn]{Question}

\def\Ric{\mathop{\rm Ric}\nolimits}

\def\Rm{\mathop{\rm Rm}\nolimits}
\def\tr{\mathop{\rm tr}\nolimits}

\def\vol{\mathop{\rm vol}\nolimits}

\def\dim{\mathop{\rm dim}\nolimits}
\def\vol{\mathop{\rm Vol}\nolimits}

\def\Isom{\mathop{\rm Isom}\nolimits}
\def\div{\mathop{\rm div}\nolimits}

\def\A{\mathop{\rm A}\nolimits}

\def\AVR{\mathop{\rm AVR}\nolimits}

\def\Li{\mathop{\rm \mathscr{L}}\nolimits}

\def\Ric{\mathop{\rm Ric}\nolimits}

\def\Rm{\mathop{\rm Rm}\nolimits}
\def\tr{\mathop{\rm tr}\nolimits}

\def\vol{\mathop{\rm vol}\nolimits}

\def\dim{\mathop{\rm dim}\nolimits}
\def\vol{\mathop{\rm Vol}\nolimits}

\def\Isom{\mathop{\rm Isom}\nolimits}
\def\div{\mathop{\rm div}\nolimits}

\def\A{\mathop{\rm A}\nolimits}

\def\AVR{\mathop{\rm AVR}\nolimits}

\def\Li{\mathop{\rm \mathscr{L}}\nolimits}
\def\HypI{\mathop{\rm (\mathscr{H}_1)}\nolimits}
\def\HypII{\mathop{\rm (\mathscr{H}_2)}\nolimits}

\def\supp{\mathop{\rm supp}\nolimits}

\def\R{\mathop{\rm R}\nolimits}
\def\Ricinf{\mathop{\rm \mathfrak{Ric}}\nolimits}
\def\Rinf{\mathop{\rm \mathfrak{R}}\nolimits}

\title[Unique continuation at infinity for conical Ricci expanders]{Unique continuation at infinity for conical Ricci expanders}

\author[Alix Deruelle]{Alix Deruelle}
\address[Alix Deruelle]{Mathematics Institute, University of Warwick, Gibbet Hill Rd, Coventry, West Midlands CV4 7AL}
\email{A.Deruelle@warwick.ac.uk}

\begin{document}
\begin{abstract}
We establish Carleman inequalities for the weighted laplacian associated to an expanding gradient Ricci soliton. As a consequence, a unique continuation at infinity is proved for asymptotically Ricci flat Ricci expanders. The obstruction at infinity is a symmetric $2$-tensor defined on the link of the corresponding asymptotic cone. 
\end{abstract}
\maketitle
\section{Introduction}
We investigate the unique continuation at infinity of conical Ricci gradient expanders. Recall that a Ricci expander $(M^n,g(\tau))_{\tau\in(-1,+\infty)}$ is a fixed point of the Ricci flow of the form  $g(\tau)=(1+\tau)\phi_{\tau}^*g$ where $\partial_{\tau}\phi_{\tau}=-\nabla^gf/(1+\tau)$, for $\tau\in(-1,+\infty)$, where $f:M\rightarrow \mathbb{R}$ is a smooth function called the potential function. Equivalently, a gradient Ricci expander is a triplet $(M^n,g,\nabla^gf)$ such that the corresponding Bakry-Émery tensor is constantly negative, i.e. $$\Ric(g)+\nabla^{g,2}(-f)=-\frac{g}{2}.$$ 

It turns out that these singularities can be used to smooth out metric cones instantaneously and appear as blow-down of non collapsed solutions to the Ricci flow with nonnegative curvature operator : \cite{Sch-Sim}. Numerous examples of asymptotically conical expanding gradient Ricci solitons have been found : they can be classified, say, in two subclasses according to their convergence rate to their asymptotic cone. Before going further, we recall below the definition of an asymptotically conical Ricci expander : see appendix \ref{sol-equ-sec} for the notations.
\begin{defn}\label{defn-asy-con-egs}
A normalized expanding gradient Ricci soliton $(M^n,g,\nabla^g f)$ is asymptotically conical with asymptotic cone $(C(X),g_{C(X)}:=dr^2+r^2g_X,r\partial r/2)$ if there exists a compact $K\subset M$, a positive radius $R$ and a diffeomorphism $\phi:M\setminus K\rightarrow C(X)\setminus B(o,R)$ such that
\begin{eqnarray}
&&\sup_{\partial B(o,r)}\arrowvert\nabla^k(\phi_*g-g_{C(X)})\arrowvert_{g_{C(X)}}=\textit{O}(f_{k}(r)),\quad \forall k\in \mathbb{N},\label{class-def-asy-con}\\
&&(f+\mu(g))(\phi^{-1}(r,x))=\frac{r^2}{4},\quad\forall (r,x)\in C(X)\setminus B(o,R), \label{cond-exp-1}
\end{eqnarray}
where $f_k(r)=\textit{o}(1)$ as $r\rightarrow+\infty$ and where $\mu(g)$ denotes the entropy.
\end{defn}

As shown in \cite{Der-Asy-Com-Egs}, either the convergence is polynomial (generic case) and the convergence rate is $\tau=2$, i.e. $f_k(r)=r^{-2-k}$ for any nonnegative integer $k$. Explicit asymptotically conical Ricci expanders are given by the rotationally symmetric examples due to Bryant [Chap. $1$,\cite{Cho-Lu-Ni-I}] coming out of the cones $(C(\mathbb{S}^{n-1}),dr^2+r^2c^2g_{\mathbb{S}^{n-1}},r\partial_r/2)_{c>0}$. In the Kähler setting, similar examples have been built by Cao \cite{Cao-Egs}. This ansatz has been extended by \cite{Fel-Ilm-Kno} where they produced Kähler Ricci expanders coming out of the cones $(C(\mathbb{S}^{2n-1}/\mathbb{Z}_k), i\partial\bar{\partial}\arrowvert\cdot\arrowvert^{2p}/p,r\partial_r/2)$ where $k>n$ (condition equivalent to negative first Chern class) and where $p$, the angle, is positive different from $1$. See \cite{Buz-Dan-Gal} and the references therein for more sophisticated examples using a similar ansatz.  Implicit deformations of the Bryant examples have been proved to exist by the author \cite{Der-Asy-Com-Egs}, \cite{Der-Smo-Pos-Met-Con} : it turns out that the uniqueness issue comes almost for free by a continuity method.

  Either the convergence is exponential (asymptotically Ricci flat case) and the convergence rate is $\tau=n$, i.e. $f_k(r)=r^{-n+k}e^{-r^2/4}$ for any nonnegative integer $k$. The first non flat asymptotically conical Ricci expanders coming out of a Ricci flat cone are the examples due to \cite{Fel-Ilm-Kno} mentioned above  with $p=1$. \cite{Fut-Wan} provided a more systematic study of Kähler Ricci expanders coming out of Kähler Ricci flat cones. Implicit examples have been built by Siepmann \cite{Sie-PHD}, where some of the previous examples are recovered : see the references therein.
  
In the asymptotically Ricci flat case, a continuity method is not adequate, moreover, it turns out that the uniqueness at infinity should not be true in general : see \cite{Ilm-Mcf-Lec} for counterexamples in the Mean Curvature Flow setting. Besides, conical Ricci expanders share many analogies with conformally compact Einstein metrics. In this regard, we benefited from the work of Biquard \cite{Biq-Uni-Con} and Anderson-Herzlich \cite{And-Her} dealing with uniqueness issues at infinity of such metrics. In this context, the obstruction at infinity is given by a symmetric $2$-tensor defined on the conformal infinity : it is a global invariant, i.e. is not an invariant depending locally on the metric at the boundary. It turns out that there is a similar obstruction in the setting of conical Ricci expanders. See definition \ref{defn-ric-tens-infty-dif} below. A closer motivation comes from the recent work of Kotschwar-Lu \cite{Kot-Lu-Con} dealing with the uniqueness at infinity of conical Ricci shrinkers : as the asymptotic cone is at the end of the lifetime of such a singularity, there is no obstruction at infinity. Consequently, their approach is based on the uniqueness of backward solutions to nonlinear evolution equations. We now start with the definition of the obstruction tensor of two Ricci expanders with isometric asymptotic cones.

\begin{defn}\label{defn-ric-tens-infty-dif}
Let $(M_i^n,g_i,\nabla^{g_i}f_i)_{i=1,2}$ be two expanding gradient Ricci solitons with isometric asymptotic cones $(C(X_i),dr^2+r^2g_{X_i},r\partial_r/2)_{i=1,2}$. Then the obstruction tensor at infinity of these two expanders is defined, whenever it makes sense, by 
\begin{eqnarray}
\Ricinf(g_2-g_1):=\lim_{r\rightarrow +\infty}r^{n-2}e^{r^2/4}\left({\phi_2}_*\Ric(g_2)-{\phi_{2}^1}_*{\phi_1}_*\Ric(g_1)\right)|_{\{r\}\times X_2},
\end{eqnarray}
where $\phi_{2}^1$ denotes an isometry (preserving the cone structure) between the two asymptotic cones. 
\end{defn}

We consider the following asymptotic condition ensuring a unique continuation result at infinity.
Let $(M^n,g,\nabla^g f)$ be an expanding gradient Ricci soliton asymptotic to $(C(X),dr^2+r^2g_X,r\partial_r/2)$. Then $\HypI$ consists of the following condition at infinity : \\

\begin{itemize}
\item Pinched Ricci curvature at infinity : $$\sup_{\{1\}\times X}\arrowvert\Ric(g_{C(X)})\arrowvert\leq \epsilon,$$ for some small positive $\epsilon$, i.e. $$\limsup_{+\infty}f\arrowvert \Ric(g)\arrowvert=:\A^0_g(\Ric(g))\leq \epsilon.$$
\end{itemize}

The first main theorem is the following :
\begin{theo}\label{Main-theo-uni-inf-0}
Let $(M_i^n,g_i,\nabla^{g_i}f_i)_{i=1,2}$ be two expanding gradient Ricci solitons with isometric asymptotic cones satisfying $\HypI$. Assume the obstruction tensor $\Ricinf(g_2-g_1)$ is well-defined and is a smooth symmetric $2$-tensor on the link $X_2$. Then $(M_1^n,g_1,\nabla^{g_1}f_1)$ and $(M_2^n,g_2,\nabla^{g_2}f_2)$ are isometric outside a compact set if the obstruction tensor $\Ricinf(g_2-g_1)$ vanishes.
\end{theo}


The second main result is actually a corollary of theorem \ref{Main-theo-uni-inf-0} and deals with asymptotically Ricci flat expanding gradient Ricci solitons.

\begin{rk}
According to \cite{Der-Asy-Com-Egs}, it turns out that if an expanding gradient Ricci soliton $(M^n,g,\nabla^g f)$ is \textit{weakly} Ricci flat, i.e. if one only assumes $\lim_{r_p\rightarrow+\infty}r_p^2\Ric=0$ where $r_p$ denotes the distance function to a fixed point of the expander, then it is asymptotically conical as defined above and the convergence to the asymptotic cone is exponential at rate $\tau=n$. 
\end{rk}

The following proposition shows that the obstruction tensor defined previously is always well-defined in the setting of asymptotically Ricci flat Ricci expanders, i.e. is a smooth tensor on the section of the asymptotic cone.

\begin{prop}\label{prop-well-def-ric-tens-infty}
Let $(M^n,g,\nabla^g f)$ be a normalized asymptotically Ricci flat expanding gradient Ricci soliton. Then the limit $$\lim_{f\rightarrow +\infty}\left(f+\mu(g)\right)^{\frac{n}{2}-1}e^{f+\mu(g)}\Ric(g),$$ exists and defines a smooth tensor on the link of the asymptotic cone.
\end{prop}

\begin{defn}
Let $(M^n,g,\nabla^g f)$ be an asymptotically Ricci flat expanding gradient Ricci soliton. Then the Ricci curvature and the scalar curvature at infinity of $(M^n,g,\nabla^g f)$ are defined by 
\begin{eqnarray}
\Ricinf(g):=\lim_{f\rightarrow +\infty}\left(f+\mu(g)\right)^{\frac{n}{2}-1}e^{f+\mu(g)}\Ric(g),\\
\Rinf(g):=\lim_{f\rightarrow +\infty}\left(f+\mu(g)\right)^{\frac{n}{2}-1}e^{f+\mu(g)}\R(g).
\end{eqnarray}
\end{defn}

Therefore, we can reformulate theorem \ref{Main-theo-uni-inf-0} in a cleaner way in the setting of asymptotically Ricci flat expanders.

\begin{theo}\label{Main-theo-uni-inf-1}
Let $(M_i^n,g_i,\nabla^{g_i}f_i)_{i=1,2}$ be two expanding gradient Ricci solitons with isometric Ricci flat asymptotic cones. Then $(M_1^n,g_1,\nabla^{g_1}f_1)$ and $(M_2^n,g_2,\nabla^{g_2}f_2)$ are isometric outside a compact set if they have the same Ricci curvature at infinity, i.e. if ${\phi_2^1}_*\Ricinf(g_1)=\Ricinf(g_2)$.
\end{theo}
By analogy with conformally compact Einstein manifolds, we ask for the dependence of the Ricci curvature at infinity on the metric $g$. We do not have a definitive answer but we compute it for some examples. In dimension $2$, one has a complete understanding of this obstruction tensor. We summarize some computations in the following lemma : 
\begin{lemma}\label{com-ric-cur-inf}
\begin{itemize}
\item If $(\Sigma^2,g,\nabla^g f)$ is a normalized asymptotically conical expanding gradient Ricci soliton. Then,
\begin{eqnarray*}
\Rinf(g)=\left(\min_Mf+\mu(g)\right)e^{\min_Mf+\mu(g)}.\\
\end{eqnarray*}
\item If $(L^{-k},g_{k})$ is the Feldman-Ilmanen-Knopf example coming out of the metric cone $(C(\mathbb{S}^{2n-1}/\mathbb{Z}_k), i\partial\bar{\partial}\arrowvert\cdot\arrowvert^2)$, where $L$ is the canonical line bundle over the complex projective space $\mathbb{CP}^{n-1}$, with $k>n$, then
\begin{eqnarray*}
\Rinf(g)=e^{\mu(g)+k-n}(k-n)^n,\quad n=\dim_{\mathbb{C}}L^{-k}.
\end{eqnarray*}

\end{itemize}
\end{lemma}

As a preliminary result to theorem \ref{Main-theo-uni-inf-0}, we investigate the sharp decay of eigentensors associated to the weighted laplacian on asymptotically conical expanders satisfying other asymptotic conditions.\\

$\HypII$ consists of one of the following conditions : \\
\begin{enumerate}
\item $(X,g_X)$ is Einstein and $\R_{g_X}>-(n-1)(n-2)$, i.e. $\R_{g_{C(X)}}>0$, \label{con-1-hyp-I}\\
\item Nonnegative Ricci curvature at infinity, i.e. $\Ric(g)\geq 0$ outside a compact set, \label{con-1-hyp-II}\\
\item $\dim M^n=n=3$ and nonnegative scalar curvature at infinity, i.e. $\R_g\geq 0$ outside a compact set. \label{con-1-hyp-III}
\end{enumerate}

The third main result of this paper is then : 

\begin{theo}\label{theo-sharp-dec-eig-tens-I}
Let $(M^n,g,\nabla^gf)$ be an asymptotically conical expanding gradient Ricci soliton satisfying $\HypI$ or $\HypII$. Assume $h$ is a smooth tensor satisfying, for some $\lambda\in\mathbb{R}$,
\begin{eqnarray}\label{ell-equ-sharp-dec-eig-tens-I}
&&-\Delta_fh-\lambda h=V_1\ast h+V_2\ast\nabla h,\quad h=\textit{o}\left(f^{\lambda-\frac{n}{2}}e^{-f}\right),
\end{eqnarray}
 where,
\begin{eqnarray*}
&& \nabla^kV_1=\textit{O}(f^{-1-k/2}), \quad k\in\{0,1,2\},\\
&& \nabla^kV_2=\textit{O}(f^{-3/2-k/2}),\quad k\in\{0,1,2\},\\
\end{eqnarray*}
and where, if $A$ and $B$ are two tensors, $A\ast B$ denotes any linear combination of contractions of the tensorial product of $A$ and $B$.
Then $h\equiv 0$.
\end{theo}

As an immediate consequence of theorem \ref{theo-sharp-dec-eig-tens-I}, one gets sharp decay of eigentensors associated to the Lichnerowicz operator acting on symmetric $2$-tensors or the weighted Hodge-DeRham laplacian acting on differential forms (also called the Witten laplacian) : 
\begin{coro}
Let $(M,g,\nabla^gf)$ be an asymptotically conical expanding gradient Ricci soliton satisfying $\HypI$ or $\HypII$. Assume $h$ is a tensor satisfying 
\begin{eqnarray*}
\Delta_fh+\Rm(g)\ast h=-\lambda h,\quad h\in L^2(e^fd\mu_g)\setminus\{0\},
\end{eqnarray*}
for some $\lambda\in\mathbb{R}$.
Then,
\begin{eqnarray*}
0<\limsup_{+\infty}f^{n/2-\lambda}e^f\arrowvert h\arrowvert\leq \| f^{n/2-\lambda}e^f h\|_{L^{\infty}}<+\infty.
\end{eqnarray*}
\end{coro}

Both proofs of theorem \ref{Main-theo-uni-inf-0} and \ref{theo-sharp-dec-eig-tens-I} are based on an adaptation due to Donnelly \cite{Don-Ale-Spe} of the proof of the "virial" theorem in quantum mechanics in a Riemannian setting. The main argument consists in establishing the so called Carleman inequalities for the weighted laplacian $\Delta_f:=\Delta+\nabla f$ associated to an expanding gradient Ricci soliton $(M^n,g,\nabla^g f)$ : see lemma \ref{lemm-cruc-C^0}. 

To give a first feeling of the proof, let us consider the eigenvalue problem 
\begin{eqnarray}
&&-\Delta_fh+\lambda h=0,\label{eigen-pb-wei-lap}\\
&&\quad h \in L^2(d\mu_f),\quad\lambda\in\mathbb{R},
\end{eqnarray}
where $d\mu_f:=e^fd\mu_g.$
Recall that $\Delta_f$ is symmetric on $L^2(d\mu_f)$. As shown in \cite{Der-Asy-Com-Egs}, a solution to (\ref{eigen-pb-wei-lap}) has mainly two possible behaviors at infinity : either it decays like $f^{-\lambda}$, or it decays like $f^{\lambda-n/2}e^{-f}$. The latter is the only solution in $L^2(e^fd\mu_g)$. Furthermore, if one assumes that $h=\textit{o}(f^{\lambda-n/2}e^{-f})$, then it turns out that there is some universal positive number $\alpha$ such that $h=\textit{O}(f^{\lambda-n/2-\alpha}e^{-f})$. To prove this fact, we derive the evolution equation of the rescaled tensor $h_{\lambda}:=f^{n/2-\lambda}e^fh$. It turns out that $h_{\lambda}$ satisfies a static backward heat-like equation, i.e. it involves the Ornstein-Uhlenbeck operator $\Delta_{-f}:=\Delta-\nabla f$ that implies automatic regularity at infinity : see section \ref{sec-pt-bd-orn-uhl}. This first step lets us start a bootstrap argument that consists in showing that $h_{\lambda}$ decays faster than polynomially in adequate $L^2$ spaces.

The main difficulty in adapting these arguments to a non linear setting is that the metric does not satisfy a non degenerate elliptic equation : this is due to the invariance of the Ricci tensor under the action by the diffeomorphism group. Nonetheless, as observed in \cite{Biq-Uni-Con} or in \cite{Kot-Lu-Con}, the Ricci tensor does satisfy a nice elliptic equation which leads to consider a system involving both the Ricci tensor and the metric : see section \ref{section-proof-main-theo}.

\begin{rk}
It turns out that the proofs of theorems \ref{Main-theo-uni-inf-0} and \ref{theo-sharp-dec-eig-tens-I} do not need $C^{\infty}$ regularity of the Ricci expander up to the boundary at infinity. Only a finite number of bounded rescaled covariant derivatives of the curvature tensor are actually needed, say at least five by inspecting the proofs. On one hand, we decided to focus on asymptotically conical expanders that are smooth up to the boundary since the setting of theorem \ref{Main-theo-uni-inf-1} automatically implies full regularity at infinity. On the other hand, it would be suitable to reduce the assumption on the smoothness in the generic case because of the existence of such expanders with only $C^{3,\theta}$ regularity at infinity : \cite{Der-Smo-Pos-Met-Con}.
\end{rk}

We end this introduction by several questions. Once the uniqueness at infinity is established, the possibility of extending Killing fields from the section of the asymptotic cone comes to mind. In the case of positively curved expanding gradient Ricci solitons, Chodosh \cite{Cho-EGS} has proved that there is a unique way to smooth out the most symmetric metric cones $(C(\mathbb{S}^{n-1}),dr^2+(cr)^2g_{\mathbb{S}^{n-1}})_{c\in(0,1]}$ by Ricci gradient expanders : as mentioned above, these expanders are provided by the rotationally symmetric Bryant examples. On the other hand, if one considers general Ricci flows $(M^n,g(t))_{t\in[0,T)}$ on smooth Riemannian manifolds with bounded curvature starting from a smooth metric $g(0)$ with bounded curvature, it is well-known that any isometry of the initial metric remains an isometry of the flow, i.e. $\Isom(M,g(0))\subset\Isom(M,g(t))$, for any $t\in[0,T)$. A much deeper result due to Kotschwar \cite{Kot-Bac-Uni} implies that the isometry group does not increase with time unless it reaches a singularity. Therefore, in the setting of expanders, we ask the following (possibly naive) question.

\begin{ques}\label{ext-sym-inf-loc}
Let $(M^n,g,\nabla^g f)$ be an expanding gradient Ricci soliton asymptotically conical to $(C(X),dr^2+r^2g_X,r\partial_r/2)$.
If $U_X$ is a Killing field on $X$, is there a Killing field $U$ asymptotic to $U_X$ on $M^n$ orthogonal to $\nabla f$. In particular, is it true that $\dim\Isom(M,g)\geq\dim\Isom(X,g_X)$ ? 
\end{ques}
A related question related to the work of Anderson and Herzlich \cite{And-Her} on conformally compact Einstein manifolds is the following. 

\begin{ques}\label{ext-sym-inf-glo}
Let $(M^n,g,\nabla^g f)$ be an expanding gradient Ricci soliton asymptotically conical to $(C(X),dr^2+r^2g_X,r\partial_r/2)$.
If $\pi_1(M,X)=\{1\}$, can any connected group of isometries of $(X,g_X)$ be extended to an action of isometries on $(M^n,g)$ ?
\end{ques}
Of course, a yes answer to the previous question would imply the uniqueness of the Bryant examples $(M^n,g_c,\nabla f_c)_{c>0}$ asymptotical to $(C(\mathbb{S}^{n-1}),dr^2+(cr)^2g_{\mathbb{S}^{n-1}},r\partial_r/2)_{c>0}$.\\

The structure of this paper is as follows. Section \ref{sec-pt-bd-orn-uhl} studies the regularity at infinity of solutions of backward heat-like equation : the main result is theorem \ref{a-prio-point-bd-orns-op} which leads to proposition \ref{prop-well-def-ric-tens-infty}. We establish Carleman inequalities for the weighted laplacian $\Delta_f$ in section \ref{sec-car-ine-wei-lap} under $\HypI$ or $\HypII$. Section \ref{sec-theo-sharp-dec-eig-tens-I} is devoted to the proof of theorem \ref{theo-sharp-dec-eig-tens-I}. Section \ref{section-proof-main-theo} proves theorem \ref{section-proof-main-theo} and establishes lemma \ref{com-ric-cur-inf}.


The author is supported by the EPSRC on a Programme Grant entitled ‘Singularities of Geometric Partial Differential Equations’ (reference number EP/K00865X/1).

\section{A priori pointwise bounds on weighted elliptic equations}\label{sec-pt-bd-orn-uhl}

In this section, we establish a regularity theorem at infinity for the Ornstein-Uhlenbeck operator.

\begin{theo}\label{a-prio-point-bd-orns-op}
Let $(M^n,g,\nabla^g f)$ be an expanding gradient Ricci soliton asymptotically conical.
Let $h$ be a smooth tensor solution to 
\begin{eqnarray*}
\Delta_fh=-\lambda h+V_0\ast h+V_1\ast \nabla h+Q,
\end{eqnarray*}
where $Q$, $V_0$, $V_1$ are smooth tensors. 
Define the rescaled solution $h_{\lambda}:=v^{\frac{n}{2}-\lambda}e^vh$ where $v$ is defined as in appendix \ref{sol-equ-sec}. Analagously, define $Q_{\lambda}:=v^{\frac{n}{2}-\lambda}e^vQ$. Assume 
\begin{eqnarray}
&&\sup_M\arrowvert h_{\lambda}\arrowvert<+\infty,\\
&&\mathscr{V}_{0,k}:=\sup_Mv^{k/2}\arrowvert\nabla^k(vV_0)\arrowvert<+\infty,\quad\forall k\geq 0,\\
&&\mathscr{V}_{1,k}:=\sup_Mv^{k/2}\arrowvert\nabla^k(v^{3/2}V_1)\arrowvert<+\infty,\quad\forall k\geq 0,\\
&&\mathfrak{Q}_{\lambda,k}:=\sup_Mv^{k/2}\arrowvert\nabla^k(vQ_{\lambda})\arrowvert<+\infty,\quad\forall k\geq 0.
\end{eqnarray}
Then,
\begin{eqnarray*}
\sup_Mv^{k/2}\arrowvert\nabla^kh_{\lambda}\arrowvert\leq C\left(k,n,g,\sup_M\arrowvert h_{\lambda}\arrowvert,(\mathfrak{Q}_{\lambda,i})_{0\leq i\leq k},(\mathscr{V}_{0,i})_{0\leq i\leq k},(\mathscr{V}_{1,i})_{0\leq i\leq k}\right),
\end{eqnarray*}
for $k\geq1$.

\end{theo}

\begin{rk}
Again, most geometric applications do not need smoothness both of the metric and the tensors $V_0$ and $V_1$ up to the boundary at infinity.
For instance, if $V_0=\Rm(g)$, $V_1=\nabla\Rm(g)$ (respectively $V_1\equiv 0$) and $Q\equiv 0$, in order to get $C^2$ a priori rescaled bounds, we need the metric to be $C^5$ (respectively $C^4$) at infinity. 

Moreover, theorem \ref{a-prio-point-bd-orns-op}
requires optimal bounds on $Q$, regarding the uniqueness issue of expanders : see the proof of theorem \ref{Main-theo-uni-inf-0}.
\end{rk}
\begin{proof}
One computes the evolution equation of $h_{\lambda}$. Similar computations appear in \cite{Der-Asy-Com-Egs}. 
\begin{eqnarray}
\Delta_fh_{\lambda}&=&\left[\Delta_f(v^{\frac{n}{2}-\lambda}e^v)-2\arrowvert\nabla \ln(v^{\frac{n}{2}-\lambda}e^v)\arrowvert^2\right]h_{\lambda}\\
&&+2<\nabla \ln(v^{\frac{n}{2}-\lambda}e^v),\nabla h_{\lambda}>+v^{\frac{n}{2}-\lambda}e^v\Delta_fh.
\end{eqnarray}
Now, by using the soliton identities given by lemma \ref{id-EGS},
\begin{eqnarray*}
&&\Delta_f(v^{\frac{n}{2}-\lambda}e^v)-2\arrowvert\nabla \ln(v^{\frac{n}{2}-\lambda}e^v)\arrowvert^2=
\left[v+\arrowvert\nabla v\arrowvert^2\left(1-2\left(1+\left(\frac{n}{2}-\lambda\right)\frac{1}{v}\right)^2\right)\right]v^{\frac{n}{2}-\lambda}e^v\\
&&+\left[(n-2\lambda)\frac{\arrowvert\nabla v\arrowvert^2}{v}+\left(\frac{n}{2}-\lambda\right)+\left(\frac{n}{2}-\lambda\right)\left(\frac{n}{2}-\lambda-1\right)\arrowvert\nabla\ln v\arrowvert^2\right]v^{\frac{n}{2}-\lambda}e^v\\
&=&\left(\lambda+\bar{V}_0\right)v^{\frac{n}{2}-\lambda}e^v,
\end{eqnarray*}
where $\bar{V}_0$ satisfies $\nabla^k\bar{V}_0=\textit{O}(v^{-1-k/2})$, for any nonnegative integer $k$. Hence,
\begin{eqnarray*}
\Delta_{-f+(2\lambda-n)\ln v}h_{\lambda}&=&(V_0+\bar{V_0}+\bar{V_1})\ast h_{\lambda}+V_1\ast\nabla h_{\lambda}+Q_{\lambda},
\end{eqnarray*}
where $Q_{\lambda}:=v^{n/2-\lambda}e^vQ$, $\bar{V_1}$ satisfies the same asymptotics as $V_0$. Finally, $h_{\lambda}$ satisfies
\begin{eqnarray}\label{ell-eq-resc-sol}
\Delta_{-f}h_{\lambda}&=&W_0\ast h_{\lambda}+W_1\ast\nabla h_{\lambda}+Q_{\lambda},
\end{eqnarray}
where $W_0$ behaves like $V_0$ at infinity and $W_1$ is such that
\begin{eqnarray*}
\mathscr{W}_{1,k}:=\sup_Mv^{k/2}\arrowvert\nabla^k (v^{1/2} W_1)\arrowvert<+\infty,\quad\forall k\geq 0.
\end{eqnarray*}

We claim that if $h_{\lambda}$ is bounded so are the weighted covariant derivatives $v^{k/2}\nabla^kh_{\lambda}$ for any integer $k$. 

Indeed, this is an adaptation of the arguments originally due to Shi and adapted to the soliton case in \cite{Der-Asy-Com-Egs}. We only prove the case $k=1$. The cases $k\geq2$ go along the same lines. Along the proof, $W_0$ and $W_1$ can change from line to line but have the same asymptotics together with their covariant derivatives as before.
\begin{eqnarray*}
\Delta_{-f}(v^{1/2}\nabla h_{\lambda})&=&v^{1/2}\nabla(\Delta_{-f}h_{\lambda})+(v^{1/2}\nabla\Rm(g))\ast h_{\lambda}+\Rm(g)\ast(v^{1/2}\nabla h_{\lambda})\\
&=&v^{1/2}\nabla(W_0\ast h_{\lambda}+W_1\ast\nabla h_{\lambda}+Q_{\lambda})+(v^{1/2}\nabla\Rm(g))\ast h_{\lambda}+\Rm(g)\ast(v^{1/2}\nabla h_{\lambda})\\
&=&W_1\ast \nabla(v^{1/2}\nabla h_{\lambda})+W_0\ast (v^{1/2}\nabla h_{\lambda}+ h_{\lambda})+v^{1/2}\nabla Q_{\lambda}.
\end{eqnarray*}
Therefore,
\begin{eqnarray*}
\Delta_{-f}\arrowvert h_{\lambda}\arrowvert^2&\gtrsim& \arrowvert \nabla h_{\lambda}\arrowvert^2-\arrowvert W_0\arrowvert\arrowvert h_{\lambda}\arrowvert^2-\arrowvert W_1\arrowvert\arrowvert \nabla h_{\lambda}\arrowvert \arrowvert h_{\lambda}\arrowvert-\arrowvert Q_{\lambda}\arrowvert\arrowvert h_{\lambda}\arrowvert\\
&\gtrsim&\arrowvert\nabla h_{\lambda}\arrowvert^2-v^{-1}\arrowvert h_{\lambda}\arrowvert^2-v\arrowvert Q_{\lambda}\arrowvert^2,\\
\Delta_{-f}\arrowvert v^{1/2}\nabla h_{\lambda}\arrowvert^2&\gtrsim& \arrowvert \nabla (v^{1/2}\nabla h_{\lambda})\arrowvert^2-\arrowvert W_1\arrowvert \arrowvert\nabla (v^{1/2}\nabla h_{\lambda})\arrowvert\arrowvert v^{1/2}\nabla h_{\lambda}\arrowvert\\
&&-\arrowvert W_0\arrowvert\left(\arrowvert v^{1/2}\nabla h_{\lambda}\arrowvert \arrowvert h_{\lambda}\arrowvert+\arrowvert v^{1/2}\nabla h_{\lambda}\arrowvert^2\right)-\arrowvert v^{1/2}\nabla Q_{\lambda}\arrowvert\arrowvert v^{1/2}\nabla h_{\lambda}\arrowvert\\
&\gtrsim&\arrowvert \nabla (v^{1/2}\nabla h_{\lambda})\arrowvert^2-v^{-1}(\arrowvert v^{1/2}\nabla h_{\lambda}\arrowvert ^2+\arrowvert  h_{\lambda}\arrowvert ^2)-v\arrowvert v^{1/2}\nabla Q_{\lambda}\arrowvert^2,
\end{eqnarray*}
where $\gtrsim$ means " not less than " up to a positive multiplicative (universal) constant.
Define $u_1:=\arrowvert h_{\lambda}\arrowvert^2$, $u_2:=v\arrowvert \nabla h_{\lambda}\arrowvert^2$. Then, consider the function $u_2(u_1+a)$ with $a$ a positive number to be defined later. One has,
\begin{eqnarray*}
\Delta_{-f}((u_1+a)u_2)&\gtrsim& \frac{u_2^2}{v}-v^{-1}u_1u_2+2<\nabla u_1,\nabla u_2>\\
&&+\arrowvert\nabla (v^{1/2}\nabla h_{\lambda})\arrowvert^2(u_1+a)-v^{-1}(u_1+u_2)(u_1+a)\\
&&-v\arrowvert Q_{\lambda}\arrowvert^2u_2-v\arrowvert v^{1/2}\nabla Q_{\lambda}\arrowvert^2(u_1+a)\\
&\gtrsim&\frac{u_2^2}{v}-v^{-1}u_1u_2-v^{-1/2}\arrowvert\nabla (v^{1/2}\nabla h_{\lambda})\arrowvert u_2u_1^{1/2}\\
&&+\arrowvert\nabla (v^{1/2}\nabla h_{\lambda})\arrowvert^2(u_1+a)-v^{-1}(u_1+u_2)(u_1+a)\\
&&-v\arrowvert Q_{\lambda}\arrowvert^2u_2-v\arrowvert v^{1/2}\nabla Q_{\lambda}\arrowvert^2(u_1+a)\\
&\gtrsim&\frac{u_2^2}{v}-v^{-1}u_1u_2+\arrowvert\nabla (v^{1/2}\nabla h_{\lambda})\arrowvert^2(-u_1+a)\\
&&-v^{-1}(u_1+u_2)(u_1+a)-v\arrowvert Q_{\lambda}\arrowvert^2u_2-v\arrowvert v^{1/2}\nabla Q_{\lambda}\arrowvert^2(u_1+a),\\
\end{eqnarray*}
for any positive $a$. Hence, by choosing $a$ proportional to $\sup_M\arrowvert h_{\lambda}\arrowvert^2=\sup_M u_1$ accordingly,
\begin{eqnarray*}
v\Delta_{-f}((u_1+a)u_2)&\gtrsim& u_2^2-(u_1+a+\sup_M\arrowvert v Q_{\lambda}\arrowvert^2)u_2-(u_1+a)\left[(u_1+a)+\sup_M\arrowvert v^{1/2}\nabla (v Q_{\lambda})\arrowvert^2\right].
\end{eqnarray*}
Consider now the following cutoff function $\psi(x):=\eta(v(x)/t)$ for some positive $t$ where $\eta:\mathbb{R}_+\rightarrow\mathbb{R}_+$ is a smooth positive function such that
\begin{eqnarray*}
\eta|_{[0,1]}\equiv 1,\quad\eta|_{[2,+\infty)}\equiv 0,\quad \frac{(\eta')^2}{\eta}\leq c,\quad\eta'\leq 0,\quad\eta''\geq-c,
\end{eqnarray*}
for some positive constant $c$. Consider the function $\psi u_2(u_1+a)$ defined on $M$ and apply the maximum principle at a point where this function attains its maximum :  after multiplying the previous inequality by $\psi^2$,
\begin{eqnarray*}
0&\gtrsim& (\psi u_2)^2-(a+\sup_M\arrowvert v Q_{\lambda}\arrowvert^2)(\psi u_2)-a\left[a+\sup_M\arrowvert v^{1/2}\nabla (v Q_{\lambda})\arrowvert^2\right]\\
&&-a\frac{v\arrowvert\nabla \psi\arrowvert^2}{\psi}(\psi u_2)+(v\Delta_{-v}\psi)(\psi u_2(u_1+a))\\
&\gtrsim&  (\psi u_2)^2-(a+\sup_M\arrowvert v Q_{\lambda}\arrowvert^2)(\psi u_2)-a\left[a+\sup_M\arrowvert v^{1/2}\nabla (v Q_{\lambda})\arrowvert^2\right],
\end{eqnarray*}
 since
\begin{eqnarray*}
\Delta_{-f}\psi&=&\eta''\frac{\arrowvert\nabla f\arrowvert^2}{t^2}+\eta'\frac{\Delta f}{t}-\eta'\frac{\arrowvert\nabla f\arrowvert^2}{t}\\
&\geq&-\frac{C}{t}.
\end{eqnarray*}
Now, as the maxima of $\psi u_2(u_1+a)$ and of $(\psi u_2)a$ are comparable, one deduces that
\begin{eqnarray*}
\sup_M\arrowvert v^{1/2}\nabla h_{\lambda}\arrowvert\leq C\left[\sup_M\arrowvert h_{\lambda}\arrowvert+\sup_M\arrowvert v Q_{\lambda}\arrowvert^2+\sup_M\arrowvert v^{1/2}\nabla( v Q_{\lambda})\arrowvert^2)\right],
\end{eqnarray*}
where $C$ is a positive constant independent of $h$ and $Q$.

\end{proof}

Proposition \ref{prop-well-def-ric-tens-infty} is a straightforward application of theorem \ref{a-prio-point-bd-orns-op} in case $(M^n,g,\nabla^g f)$ is an asymptotically Ricci flat expanding gradient Ricci soliton. Indeed, apply theorem \ref{a-prio-point-bd-orns-op} to equation (\ref{equ:6}) satisfied by $h:=\Ric(g)$ where $\lambda=1$, $V_0:=\Rm(g)$, $V_1\equiv 0$ and $Q\equiv 0$.

\section{ Carleman inequalities for the weighted laplacian}\label{sec-car-ine-wei-lap}

We derive first general algebraic commutator identities for the weighted laplacian in the spirit of \cite{Don-Ale-Spe}.

Let $(M^n,g)$ be a complete Riemannian manifold and let $f:M\rightarrow\mathbb{R}$ be a smooth function.

Let $$A=\nabla_{\nabla\phi}+\frac{\Delta_f\phi}{2}.$$ Then $A$ is antisymmetric with respect to the measure $e^fd\mu(g)$ . Let $H:=-\Delta_f$.
\begin{lemma}\label{lemma-comm-est-don}
Let $T$ be a smooth tensor. Then,
\begin{eqnarray}\label{com-id}
[H,A]T&=&-2<\nabla^2\phi,\nabla^2T>-\frac{\Delta_f(\Delta_f\phi)}{2}T-2\nabla_{\Delta_f\nabla\phi} T\\
&&+\Rm(g)(\nabla \phi,\cdot,\cdot,\cdot)\ast \nabla T+\nabla\Rm(g)\ast\nabla\phi\ast T,
\end{eqnarray}
where 
\begin{eqnarray*}
<\nabla^2\phi,\nabla^2T>:=\nabla^2_{ij}\phi\nabla^2_{ij}T.
\end{eqnarray*}

\end{lemma}

\begin{proof}
On one hand,
\begin{eqnarray*}
HAT&=&-\Delta_f\left[\nabla_{\nabla \phi} T+\frac{(\Delta_f\phi)}{2}T\right]\\
&=&-2<\nabla^2\phi,\nabla^2T>-\nabla_{\Delta_f\nabla\phi} T-<\nabla\phi,\Delta_f\nabla T>\\
&&-\frac{\Delta_f(\Delta_f\phi)}{2}T-\frac{\Delta_f\phi}{2}\Delta_fT-\nabla_{\nabla\Delta_f\phi} T,\\
\end{eqnarray*}
where $<\nabla^2\phi,\nabla^2T>:=\nabla^2_{ij}\phi\nabla^2_{ij}T,$ and $<\nabla\phi,\Delta_f\nabla T>:=\nabla_i\phi\Delta_f\nabla_i T.$
On the other hand,
\begin{eqnarray*}
AHT=\nabla_{\nabla\phi}(-\Delta_fT)-\frac{\Delta_f\phi}{2}\Delta_fT.
\end{eqnarray*}
Therefore,
\begin{eqnarray*}
[H,A]T&=&-2<\nabla^2\phi,\nabla^2T>-\frac{\Delta_f(\Delta_f\phi)}{2}T-2\nabla_{\nabla\Delta_f\phi} T\\
&&-\nabla_{\Ric_f(\nabla\phi)} T-<\nabla\phi,[\Delta_f,\nabla] T>,
\end{eqnarray*}
where $\Ric_f:=\Ric(g)+\nabla^{g,2}(-f).$
Indeed,
\begin{eqnarray*}
\Delta_f\nabla \phi&=&\nabla\Delta \phi+\Ric(g)(\nabla \phi)+\nabla_{\nabla f}\nabla \phi=\nabla(\Delta \phi+\nabla_{\nabla f}\phi)+\Ric_f(\nabla \phi),\\
\left[\Delta_f,\nabla\right]T&=&\Ric_f(\nabla T)+\Rm(g)\ast \nabla T+\nabla\Rm(g)\ast T,\\
\Ric_f(\nabla T)_i&:=&{\Ric_f}_{ik}\nabla_kT,\\
\left<\nabla \phi,\left[\Delta_f,\nabla\right]T\right>&=&\nabla_{\Ric_f(\nabla\phi)} T+\Rm(g)(\nabla \phi,\cdot,\cdot,\cdot)\ast\nabla T+\nabla\Rm(g)\ast \nabla \phi\ast T.
\end{eqnarray*}
Therefore,
\begin{eqnarray*}
[H,A]T&=&-2<\nabla^2\phi,\nabla^2T>-\frac{\Delta_f(\Delta_f\phi)}{2}T-2\nabla_{\Delta_f\nabla\phi} T\\
&&+\Rm(g)(\nabla \phi,\cdot,\cdot,\cdot)\ast\nabla T+\nabla\Rm(g)\ast\nabla\phi\ast T.
\end{eqnarray*}

\end{proof}
Now, we integrate the result of lemma \ref{lemma-comm-est-don} to get a so called Mourre estimate.

\begin{lemma}\label{Mourre-type-est}
Let $T$ be a smooth tensor on $M$ with compact support. Then, the following estimate holds : 
\begin{eqnarray*}
<[H,A]T,T>_{L^2_f}&=&\int_{M}2\nabla^2\phi(\nabla T,\nabla T)+\frac{1}{2}<\nabla\Delta_f\phi,\nabla\arrowvert T\arrowvert^2>d\mu_f\\
&&+\int_M\Rm(g)(\nabla\phi,\cdot,\cdot,\cdot)\ast \nabla T\ast T+\nabla\Rm(g)\ast\nabla\phi\ast T^{*2} d\mu_f,
\end{eqnarray*}
where
\begin{eqnarray*}
\nabla^2\phi(\nabla T,\nabla T):=\nabla^2_{ij}\phi\left<\nabla_iT,\nabla_jT\right>.
\end{eqnarray*}

\end{lemma}

\begin{proof}
By integrating the identity (\ref{com-id}), we get
\begin{eqnarray*}
<[H,A]T,T>_{L^2_f}&=&\int_{M}-2\left<T,<\nabla^2\phi,\nabla^2T>\right>-\frac{\Delta_f(\Delta_f\phi)}{2}\arrowvert T\arrowvert^2d\mu_f\\
&&-\int_{M}<\Delta_f\nabla\phi,\nabla \arrowvert T\arrowvert^2>+\Rm(g)(\nabla\phi,\cdot,\cdot,\cdot)\ast \nabla T\ast T d\mu_f\\
&&+\int_M\nabla\Rm(g)\ast\nabla \phi\ast T^{*2} d\mu_f.\\
\end{eqnarray*}
Now, applying the Stokes theorem to $\nabla^2\phi\left(\nabla\arrowvert T\arrowvert^2\right)$ with respect to the weighted measure $d\mu_f$,
\begin{eqnarray*}
-2\int_{M}\left<T,<\nabla^2\phi,\nabla^2T>\right> d\mu_f&=&\int_{M}<\Delta_f\nabla\phi,\nabla\arrowvert T\arrowvert^2>+2\nabla^2\phi(\nabla T,\nabla T)d\mu_f.
\end{eqnarray*}
Hence the result.
\end{proof}
By applying the previous lemma to $\phi=f$ together with the soliton identity (\ref{equ:4}), we get the
\begin{coro}\label{coro-mourre-est-egs}
Let $(M^n,g,\nabla^g f)$ be an expanding gradient Ricci soliton. Then, if $T$ is a smooth tensor with compact support,
\begin{eqnarray*}
&&<[H,A]T,T>_{L^2_f}=\int_{M}<H T,T>-\frac{v}{2}\arrowvert T\arrowvert^2d\mu_f\\
&&+\int_M2\Ric(g)(\nabla T,\nabla T)+\div\Rm(g)\ast\nabla T\ast T +\nabla\Rm(g)\ast\nabla f\ast T^{*2} d\mu_f.\\
\end{eqnarray*}
\end{coro}

Define for $\alpha\geq0$ and some real number $\lambda$,
\begin{eqnarray*}
F_{\alpha}:=\frac{v}{2}+\frac{\alpha-2\lambda+n/2}{2}\ln v.
\end{eqnarray*}

\begin{defn}
Define for any nonnegative integer $k$, and any tensor $h$ (with compact support),
\begin{eqnarray*}
I_{\alpha}^k(h):=I_{\alpha}^0(v^{k/2}\nabla^k h)=\int_Mv^k\arrowvert \nabla^kh\arrowvert^2e^{2F_{\alpha}}d\mu_f,\\
J^0_{\alpha}(h):=\int_M\arrowvert\Delta_fh\arrowvert^2e^{2F_{\alpha}}d\mu_f.
\end{eqnarray*}
\end{defn}

The next proposition controls the weighted integral norms of the two first covariant derivatives  of a tensor $h$ with compact support in terms of the weighted integral norms of $h$ and $\Delta_fh$.
 
\begin{prop}\label{prop-cov-der-est-wei-nor}
Let $h$ be a compactly supported smooth tensor. Then,
\begin{eqnarray*}
&&\int_M\arrowvert\nabla h\arrowvert^2e^{2F_{\alpha}}d\mu_f=I^1_{\alpha-1}(h)\lesssim J^0_{\alpha}(h)+I^0_{\alpha+1}(h)+\alpha I^0_{\alpha}(h)+\alpha^2 I^0_{\alpha-1}(h),\\
&&\int_M\arrowvert\nabla^2h\arrowvert^2e^{2F_{\alpha}}d\mu_f=I^2_{\alpha-2}(h)\lesssim J^0_{\alpha+1}(h)+\alpha J^0_{\alpha}(h)+\alpha^2 J^0_{\alpha-1}(h)+\sum_{i=-2}^2\alpha^{2-i}I^0_{\alpha+i}(h).
\end{eqnarray*}

\end{prop}
\begin{proof}By integration by parts,
\begin{eqnarray*}
\int_M<-\Delta_fh,h>e^{2F_{\alpha}}d\mu_f&=&\int_M\left<\nabla h,\nabla (e^{2F_{\alpha}}h)\right>d\mu_f\\
&=&\int_M\arrowvert\nabla h\arrowvert^2e^{2F_{\alpha}}d\mu_f+\int_M\left<\frac{\nabla \arrowvert h\arrowvert^2}{2},\nabla e^{2F_{\alpha}}\right>d\mu_f\\
&=&\int_M\arrowvert\nabla h\arrowvert^2e^{2F_{\alpha}}d\mu_f-\int_M\frac{ \arrowvert h\arrowvert^2}{2}\Delta_f e^{2F_{\alpha}}d\mu_f,
\end{eqnarray*}
which implies,
\begin{eqnarray*}
\int_M\arrowvert\nabla h\arrowvert^2e^{2F_{\alpha}}d\mu_f&\lesssim& \int_M\arrowvert\Delta_fh\arrowvert^2e^{2F_{\alpha}}d\mu_f+I^0_{\alpha+1}(h)+\alpha I^0_{\alpha}(h)+\alpha^2 I^0_{\alpha-1}(h)\\
&\lesssim&J^0_{\alpha}(h)+I^0_{\alpha+1}(h)+\alpha I^0_{\alpha}(h)+\alpha^2 I^0_{\alpha-1}(h).
\end{eqnarray*}

Now,
\begin{eqnarray*}
\Delta_f\arrowvert\nabla h\arrowvert^2&=&2\arrowvert\nabla^2h\arrowvert^2+2<\Delta_f\nabla h,\nabla h>\\
&=&2\arrowvert\nabla^2h\arrowvert^2+2<[\Delta_f,\nabla] h,\nabla h>+<\nabla\Delta_fh,\nabla h>\\
&=&2\arrowvert\nabla^2h\arrowvert^2+< (-1+\textit{O}(v^{-1}))\nabla h+\textit{O}(v^{-3/2})h,\nabla h>+<\nabla\Delta_fh,\nabla h>.
\end{eqnarray*}
Therefore,
\begin{eqnarray*}
\int_M\arrowvert\nabla^2h\arrowvert^2e^{2F_{\alpha}}d\mu_f&\lesssim&\int_M\left(v+\alpha+\frac{\alpha^2}{v}\right)\arrowvert\nabla h\arrowvert^2e^{2F_{\alpha}}d\mu_f+\int_M\arrowvert\Delta_fh\arrowvert^2e^{2F_{\alpha}}d\mu_f+I^0_{\alpha}(h)\\
&\lesssim&\left[J^0_{\alpha+1}(h)+I^0_{\alpha+2}(h)+\alpha I^0_{\alpha+1}(h)+\alpha^2 I^0_{\alpha}(h)\right]\\
&&+\alpha\left[J^0_{\alpha}(h)+I^0_{\alpha+1}(h)+\alpha I^0_{\alpha}(h)+\alpha^2 I^0_{\alpha-1}(h)\right]\\
&&+\alpha^2\left[J^0_{\alpha-1}(h)+I^0_{\alpha}(h)+\alpha I^0_{\alpha-1}(h)+\alpha^2 I^0_{\alpha-2}(h)\right]\\
&&+J^0_{\alpha}(h)+I^0_{\alpha}(h)\\
&\lesssim&J^0_{\alpha+1}(h)+\alpha J^0_{\alpha}(h)+\alpha^2 J^0_{\alpha-1}(h)+\sum_{i=-2}^2\alpha^{2-i}I^0_{\alpha+i}(h).
\end{eqnarray*}

\end{proof}
The next lemma is the key result to prove unique continuation results at infinity for Ricci expanders.

\begin{lemma}\label{lemm-cruc-C^0}
Let $(M^n,g,\nabla^g f)$ be an expanding gradient Ricci soliton asymptotically conical satisfying $\HypI$ (respectively $\HypII$). Then there exist positive constants $c_1$ and $c_2$ and a compact set $K\subset M$ such that for $\alpha\gtrsim \A^0_g(\Ric(g))$ (respectively for any positive $\alpha$) and for any tensor $h$ compactly supported outside $K$,
\begin{eqnarray*}
\int_M\left(c_1\alpha+\frac{c_1^2\alpha^2}{v^2}\right)\arrowvert h\arrowvert^2e^{2F_{\alpha}}d\mu_f\leq \left(1+\frac{c_2}{\alpha}\right) \|[(H-\lambda)h]e^{F_{\alpha}}\|^2_{L^2_f},
\end{eqnarray*}
i.e.
\begin{eqnarray*}
\alpha I^0_{\alpha}(h)+\alpha^2 I^0_{\alpha-2}(h)\lesssim (1+\alpha^{-1}) \|[(H-\lambda)h]e^{F_{\alpha}}\|^2_{L^2_f}.
\end{eqnarray*}

\end{lemma}

\begin{proof}
Consider $T_{\alpha}:=e^{F_{\alpha}}h$ where $h$ is a tensor compactly supported outside a sufficiently large compact set (independent of $\alpha$). 
First of all, applying corollary \ref{coro-mourre-est-egs} to $T_{\alpha}$ gives
\begin{eqnarray*}
&&<[H,A]T_{\alpha},T_{\alpha}>_{L^2_f}=\int_{M}<H T_{\alpha},T_{\alpha}>-\frac{v}{2}\arrowvert T_{\alpha}\arrowvert^2d\mu_f\\
&&+2\int_M\Ric(g)(\nabla T_{\alpha},\nabla T_{\alpha})+\div\Rm(g)\ast\nabla T_{\alpha}\ast T_{\alpha} +\nabla\Rm(g)\ast\nabla f\ast T_{\alpha}^{*2} d\mu_f.\\
\end{eqnarray*}

Define as in \cite{Don-Ale-Spe},
\begin{eqnarray*}
\nabla F_{\alpha}=\left(\frac{1}{2}+\frac{\alpha+n/2-2\lambda}{2v}\right)\nabla v=:w_{\alpha}\nabla v.
\end{eqnarray*}
Note that $w_{\alpha}>1/4$ for any nonnegative $\alpha$, outside a compact set independent of $\alpha$ but depending on $\lambda$.
\begin{eqnarray*}
H T_{\alpha}&=&(Hh)e^{F_{\alpha}}-2\nabla_{\nabla F_{\alpha}} T_{\alpha}+2\arrowvert\nabla F_{\alpha}\arrowvert^2T_{\alpha}-(\Delta_fe^{F_{\alpha}})h\\
&=&(Hh)e^{F_{\alpha}}+\arrowvert\nabla F_{\alpha}\arrowvert^2T_{\alpha}-2w_{\alpha}AT_{\alpha}-<\nabla f,\nabla w_{\alpha}>T_{\alpha}.
\end{eqnarray*}
Note that the operator $BT_{\alpha}:=2w_{\alpha}AT_{\alpha}+<\nabla f,\nabla w_{\alpha}>T_{\alpha}$ is antisymmetric.

On the other hand,
\begin{eqnarray*}
&&2<HT_{\alpha},AT_{\alpha}>_{L^2_f}\leq2 \left<AT_{\alpha},(Hh)e^{F_{\alpha}}+\arrowvert\nabla F_{\alpha}\arrowvert^2T_{\alpha}-<\nabla f,\nabla w_{\alpha}>T_{\alpha}\right>_{L^2_f}-\arrowvert AT_{\alpha}\arrowvert_{L^2_f},\\
&&2\left<AT_{\alpha},\arrowvert\nabla F_{\alpha}\arrowvert^2T_{\alpha}-<\nabla f,\nabla w_{\alpha}>T_{\alpha}\right>_{L^2_f}=<T_{\alpha},\nabla f\cdot(\nabla f\cdot\nabla w_{\alpha}-\arrowvert\nabla F_{\alpha}\arrowvert^2)T_{\alpha}>_{L^2_f}.
\end{eqnarray*}
Therefore,
\begin{eqnarray*}
2<HT_{\alpha},AT_{\alpha}>_{L^2_f}\leq \|[(H-\lambda)h]e^{F_{\alpha}}\|^2_{L^2_f}+<T_{\alpha},\nabla f\cdot(\nabla f\cdot\nabla w_{\alpha}-\arrowvert\nabla F_{\alpha}\arrowvert^2)T_{\alpha}>_{L^2_f}.
\end{eqnarray*}
Finally,
\begin{eqnarray*}
&&\int_M<Hh,h>e^{2F_{\alpha}}+\left(\arrowvert\nabla F_{\alpha}\arrowvert^2+\nabla f\cdot(\arrowvert\nabla F_{\alpha}\arrowvert^2-\nabla f\cdot\nabla w_{\alpha})-\frac{v}{2}\right)\arrowvert T_{\alpha}\arrowvert^2d\mu_f\leq\\
&&  \|[(H-\lambda)h]e^{F_{\alpha}}\|^2_{L^2_f}-2\int_M\Ric(g)(\nabla T_{\alpha},\nabla T_{\alpha})d\mu_f\\
&&+c(n)\int_M\arrowvert\nabla\Rm(g)\arrowvert\arrowvert\nabla T_{\alpha}\arrowvert\arrowvert T_{\alpha}\arrowvert +\arrowvert\nabla\Rm(g)\arrowvert\arrowvert\nabla f\arrowvert\arrowvert T_{\alpha}\arrowvert^{2} d\mu_f.
\end{eqnarray*}

We compute the left hand term as follows by using the soliton identities from lemma \ref{id-EGS} :
\begin{eqnarray*}
&&\nabla F_{\alpha}=\frac{1}{2}\left(1+\frac{c(\alpha,n,\lambda)}{v}\right)\nabla v,\quad c(\alpha,n,\lambda):=\alpha+\frac{n}{2}-2\lambda,\\
&&\arrowvert\nabla F_{\alpha}\arrowvert^2=\frac{1}{4}\left[1+\frac{2c(\alpha,n,\lambda)}{v}+\frac{c(\alpha,n,\lambda)^2}{v^2}\right]\arrowvert\nabla f\arrowvert^2,\\
&&\nabla f\cdot\arrowvert\nabla F_{\alpha}\arrowvert^2=\frac{\nabla^2f(\nabla f,\nabla f)}{2}\left[1+\frac{2c(\alpha,n,\lambda)}{v}+\frac{c(\alpha,n,\lambda)^2}{v^2}\right]-\frac{c(\alpha,n,\lambda)\arrowvert\nabla f\arrowvert^4}{2v^2}\left[1+\frac{c(\alpha,n,\lambda)}{v}\right],\\
&&\nabla f\cdot\nabla f\cdot\nabla w_{\alpha} =-\frac{c(\alpha,n,\lambda)}{2}\nabla f\cdot\left(\frac{\arrowvert\nabla f\arrowvert^2}{v^2}\right)=-\frac{c(\alpha,n,\lambda)}{v^2}\left[\nabla^2f(\nabla f,\nabla f)-\frac{\arrowvert\nabla f\arrowvert^4}{v}\right]\\
&&\arrowvert\nabla F_{\alpha}\arrowvert^2+\nabla f\cdot(\arrowvert\nabla F_{\alpha}\arrowvert^2-\nabla f\cdot\nabla w_{\alpha})-\frac{v}{2}=\arrowvert\nabla f\arrowvert^2\frac{1+\Ric(\mathbf{n},\mathbf{n})}{2}\left[1+\frac{2c(\alpha,n,\lambda)}{v}+\frac{c(\alpha,n,\lambda)^2}{v^2}\right]\\
&&-\frac{c(\alpha,n,\lambda)\arrowvert\nabla f\arrowvert^4}{2v^2}\left[1+\frac{c(\alpha,n,\lambda)}{v}\right]+\arrowvert\nabla f\arrowvert^2\frac{c(\alpha,n,\lambda)}{v^2}\left[\frac{1}{2}+\Ric(\mathbf{n},\mathbf{n})-\frac{\arrowvert\nabla f\arrowvert^2}{v}\right]-\frac{(\arrowvert\nabla f\arrowvert^2+\Delta f)}{2}\\
&&=\frac{c(\alpha,n,\lambda)\arrowvert\nabla f\arrowvert^2}{v}\left(1-\frac{\arrowvert\nabla f\arrowvert^2}{2v}\right)-\frac{n}{4}+\textit{O}(v^{-1})\\
&&+\frac{c(\alpha,n,\lambda)\arrowvert\nabla f\arrowvert^2}{v^2}\left(\frac{c(\alpha,n,\lambda)}{2}-\frac{c(\alpha,n,\lambda)\arrowvert\nabla f\arrowvert^2}{2v}+\left(\frac{1}{2}-\frac{\arrowvert\nabla f\arrowvert^2}{v}\right)\right)\\
&&+\textit{O}(v^{-2})\arrowvert\nabla f\arrowvert^2\left[1+\frac{2c(\alpha,n,\lambda)}{v}+\frac{c(\alpha,n,\lambda)^2+c(\alpha,n,\lambda)}{v^2}\right]\\
&&=\frac{c(\alpha,n,\lambda)}{2}\left(1-\frac{\Delta f}{v}\right)\left(1+\frac{\Delta f}{v}\right)-\frac{n}{4}+\textit{O}(v^{-1})\\
&&+\frac{c(\alpha,n,\lambda)\arrowvert\nabla f\arrowvert^2}{v^2}\left(\frac{c(\alpha,n,\lambda)}{2}\frac{\Delta f}{v}-\left(\frac{1}{2}-\frac{\Delta f}{v}\right)\right)\\
&&+\textit{O}(v^{-2})\arrowvert\nabla f\arrowvert^2\left[1+\frac{2c(\alpha,n,\lambda)}{v}+\frac{c(\alpha,n,\lambda)^2+c(\alpha,n,\lambda)}{v^2}\right]\\
&&=\frac{c(\alpha,n,\lambda)}{2}-\frac{n}{4}+(c(\alpha,n,\lambda)+1)\textit{O}(v^{-1})+\frac{c(\alpha,n,\lambda)^2\arrowvert\nabla f\arrowvert^2\Delta f}{2v^3}+c(\alpha,n,\lambda)^2\textit{O}(v^{-3}),\\
&&=-\lambda+\frac{\alpha}{2}(1+\textit{O}(v^{-1}))+\frac{n}{4}\frac{\alpha^2}{v^2}(1+\textit{O}(v^{-1})).
\end{eqnarray*}

Therefore, there exists a compact set $K\subset M$ such that for any positive $\alpha$ and for any tensor compactly supported outside $K$,
\begin{eqnarray*}
&&\int_M<(H-\lambda)h,h>e^{2F_{\alpha}}d\mu_f+\alpha I^0_{\alpha}+\frac{n}{4}\alpha^2I^0_{\alpha-2}(h)\lesssim\|[(H-\lambda)h]e^{F_{\alpha}}\|^2_{L^2_f}\\&&+\int_M-2\Ric(g)(\nabla T_{\alpha},\nabla T_{\alpha})+\arrowvert\nabla\Rm(g)\arrowvert\arrowvert\nabla T_{\alpha}\arrowvert\arrowvert T_{\alpha}\arrowvert +\arrowvert\nabla\Rm(g)\arrowvert\arrowvert\nabla f\arrowvert\arrowvert T_{\alpha}\arrowvert^{2} d\mu_f.
\end{eqnarray*}
or, by the Cauchy-Schwarz inequality,
\begin{eqnarray*}
&&\alpha I^0_{\alpha}+\alpha^2I^0_{\alpha-2}(h)\lesssim \left(1+\alpha^{-1}\right) \|[(H-\lambda)h]e^{F_{\alpha}}\|^2_{L^2_f}\\
&&+\int_M-2\Ric(g)(\nabla T_{\alpha},\nabla T_{\alpha})+\arrowvert\nabla\Rm(g)\arrowvert\arrowvert\nabla T_{\alpha}\arrowvert\arrowvert T_{\alpha}\arrowvert +\arrowvert\nabla\Rm(g)\arrowvert\arrowvert\nabla f\arrowvert\arrowvert T_{\alpha}\arrowvert^{2} d\mu_f.
\end{eqnarray*}

Now, we handle the term involving $\Ric(g)(\nabla T_{\alpha},\nabla T_{\alpha})$ by using $\HypI$ or $\HypII$ as follows : 
Assume $\HypI$ holds. Then,
\begin{eqnarray*}
\int_M\arrowvert\Ric(g)\arrowvert\arrowvert\nabla T_{\alpha}\arrowvert^{2}d\mu_f&\lesssim& \A^0_g(\Ric(g))\int_M\arrowvert\nabla T_{\alpha}\arrowvert^2v^{-1}d\mu_f\\
&\lesssim&\A^0_g(\Ric(g))\int_M\left[\arrowvert\nabla h\arrowvert^2+\arrowvert\nabla F_{\alpha}\arrowvert^2\arrowvert h\arrowvert^2+\left<\nabla\arrowvert h\arrowvert^2,\nabla F_{\alpha}\right>\right]e^{2F_{\alpha}}v^{-1}d\mu_f\\
&=&\A^0_g(\Ric(g))\int_M\left[\arrowvert\nabla h\arrowvert^2+\arrowvert \nabla F_{\alpha}\arrowvert^2\arrowvert h\arrowvert^2\right]e^{2F_{\alpha-1}}d\mu_f\\
&&-\A^0_g(\Ric(g))\frac{1}{2}\int_M\arrowvert h\arrowvert^2\Delta_{f-\ln v} \left(e^{2F_{\alpha}}\right)v^{-1}d\mu_f\\
&=&\A^0_g(\Ric(g))\int_M\left[\arrowvert\nabla h\arrowvert^2-(\arrowvert \nabla F_{\alpha}\arrowvert^2+\Delta_{f-\ln v}F_{\alpha})\arrowvert h\arrowvert^2\right]e^{2F_{\alpha-1}}d\mu_f.
\end{eqnarray*}

Hence, by proposition \ref{prop-cov-der-est-wei-nor}, we get
\begin{eqnarray}
\int_M\arrowvert\Ric(g)\arrowvert\arrowvert\nabla T_{\alpha}\arrowvert^{2}d\mu_f&\lesssim&J^0_{\alpha-1}(h)+\A^0_g(\Ric(g))\left( I^0_{\alpha}(h)+\alpha^2 I^0_{\alpha-2}(h)\right)+\alpha I^0_{\alpha-1}(h)\\
&\lesssim&J^0_{\alpha-1}(h)+\epsilon\left( I^0_{\alpha}(h)+\alpha^2 I^0_{\alpha-2}(h)\right)+\alpha I^0_{\alpha-1}(h),\label{ineq-ric-curv-term}
\end{eqnarray}
where $\epsilon\geq \A^0_g(\Ric(g))$.

Assume $\HypII$ holds.
\begin{itemize}
\item If (\ref{con-1-hyp-I}) is satisfied then in particular, as $(X,g_X)$ is Einstein, according to \cite{Der-Asy-Com-Egs}, $T(g):=\Ric(g)-\frac{\R_g}{n-1}(g-\mathbf{n}\otimes\mathbf{n})=\textit{O}(v^{-2})$. Now,  as $\R_g\geq 0$ outside a compact set,
\begin{eqnarray*}
\int_M\Ric(g)(\nabla T_{\alpha},\nabla T_{\alpha})d\mu_f&\geq& \int_MT(g)(\nabla T_{\alpha},\nabla T_{\alpha})d\mu_f\\
&\gtrsim& -\int_Mv^{-2}\arrowvert\nabla T_{\alpha}\arrowvert^2d\mu_f.
\end{eqnarray*}
The same argument works in case (\ref{con-1-hyp-III}) thanks to \cite{Der-Asy-Com-Egs}. In any case, one has a similar estimate similar to (\ref{ineq-ric-curv-term}) with $\epsilon$ arbitrarily small.
\item Assume (\ref{con-1-hyp-II}) holds then, obviously,
\begin{eqnarray*}
\int_M\Ric(g)(\nabla T_{\alpha},\nabla T_{\alpha})d\mu_f&\geq&0,
\end{eqnarray*}
if $T_{\alpha}$ is supported outside a sufficiently large compact set of $M$.
  
\end{itemize}

Similarly, using the behavior of the curvature tensor at infinity,
\begin{eqnarray*}
\int_M\arrowvert\nabla\Rm(g)\arrowvert\arrowvert\nabla T_{\alpha}\arrowvert\arrowvert T_{\alpha}\arrowvert  d\mu_f&\lesssim& \int_M\arrowvert\nabla T_{\alpha}\arrowvert^2e^{2F_{\alpha-2}}d\mu_f+I_{\alpha-1}^0(h)\\
&\lesssim& \epsilon\int_M\arrowvert\nabla T_{\alpha}\arrowvert^2e^{2F_{\alpha-1}}d\mu_f+I_{\alpha-1}^0(h)\\
&\lesssim&J^0_{\alpha-1}(h)+\epsilon I^0_{\alpha}+\alpha I^0_{\alpha-1}(h)+\epsilon\alpha^2 I^0_{\alpha-2}(h),\\
\int_M\arrowvert\nabla\Rm(g)\arrowvert\arrowvert\nabla f\arrowvert\arrowvert T_{\alpha}\arrowvert^{2}d\mu_f&\lesssim&I_{\alpha-1}^0(h),
\end{eqnarray*}
where, again, $\epsilon$ can be chosen arbitrarily small.

Therefore, there exists a compact set $K\subset M$ such that for any positive $\alpha$ (if $\HypII$ holds) or for any positive $\alpha$ large enough compared to $\epsilon\geq\A^0_g(\Ric(g))$ (if $\HypI$ holds) and for any tensor compactly supported outside $K$,

\begin{eqnarray*}
\alpha I^0_{\alpha}+\alpha^2I^0_{\alpha-2}(h)\lesssim \left(1+\alpha^{-1}\right) \|[(H-\lambda)h]e^{F_{\alpha}}\|^2_{L^2_f}.
\end{eqnarray*}

\end{proof}

We end this section by estimating the semi norms $\left(I^k_{\alpha}(h)\right)_{\alpha}^k$ of a tensor $h$, compactly supported at infinity, by the semi norms $\left(I^k_{\alpha}(\nabla_{\nabla f}h)\right)_{\alpha}^k$.

\begin{prop}\label{prop-control-rad-der}
For any nonnegative integer $k$, and  any smooth tensor $h$ such that $\supp(h)\subset M\setminus K$, where $K$ might depend on $k$,
\begin{eqnarray}
&&\sum_{i=0}^kI^i_{\alpha+1}(h)\lesssim \sum_{i=0}^kI^{i}_{\alpha-1}(\nabla_{\nabla f}h),\quad\forall \alpha\in \mathbb{R}_+,\label{est-sec-fun-form-I}\\
&&\sum_{i=0}^k\alpha^2I^i_{\alpha-1}(h)+\alpha I^i_{\alpha}(h)+I^i_{\alpha+1}(h)\lesssim \sum_{i=0}^kI^{i}_{\alpha-1}(\nabla_{\nabla f}h), \label{est-sec-fun-form-II}
\end{eqnarray}
for $\alpha$ large enough, independent of $h$.
\end{prop}

\begin{proof}
By integration by parts,
\begin{eqnarray*}
-\int_M<\nabla_{\nabla f}h,h>e^{2F_{\alpha}}d\mu_f&=&-\int_M\left<\nabla\frac{\arrowvert h\arrowvert^2}{2},\nabla f\right>e^{2F_{\alpha}}d\mu_f\\
&=&\int_M\frac{\arrowvert h\arrowvert^2}{2}\div_f(e^{2F_{\alpha}}\nabla f)d\mu_f=\int_M\frac{\arrowvert h\arrowvert^2}{2}\left(v+2\nabla_{\nabla f} F_{\alpha}\right)e^{2F_{\alpha}}d\mu_f\\
&=&\int_M\frac{\arrowvert h\arrowvert^2}{2}\left(v+\left(1+\frac{\alpha+n/2-2\lambda}{v}\right)\arrowvert\nabla f\arrowvert^2\right)e^{2F_{\alpha}}d\mu_f\\
&=&\int_M\left(v+\alpha(1+\textit{O}(v^{-1})\right)\arrowvert h\arrowvert^2e^{2F_{\alpha}}d\mu_f\geq I^0_{\alpha+1}(h)+\frac{\alpha}{2} I^0_{\alpha}(h).
\end{eqnarray*}
On the other hand,
\begin{eqnarray*}
-\int_M<\nabla_{\nabla f}h,h>e^{2F_{\alpha}}d\mu_f\leq \frac{1}{2}\left(I^0_{\alpha+1}(h)+I^0_{\alpha-1}(\nabla_{\nabla f}h)\right),\\
-\int_M<\nabla_{\nabla f}h,h>e^{2F_{\alpha}}d\mu_f\leq (I^0_{\alpha}(h)I^0_{\alpha}(\nabla_{\nabla f}h))^{1/2}.
\end{eqnarray*}
Hence the result for $k=0$. Then we proceed by induction by using commutation identities that hold on expanding gradient Ricci solitons : 
\begin{eqnarray*}
[\nabla_{\nabla f},\nabla^k]h=-\left(\frac{k}{2}+\Ric(g)\ast\right)\nabla^kh+\sum_{i=1}^k\nabla^{k-i}h\ast\nabla^i\Ric(g).
\end{eqnarray*}
In particular, if the geometry at infinity is conical,
\begin{eqnarray*}
\nabla_{\nabla f}(v^{k/2}\nabla^kh)=\nabla^k(v^{k/2}\nabla_{\nabla f}h)+\sum_{i=0}^{k}\textit{O}(v^{-1})(v^{i/2}\nabla^ih).
\end{eqnarray*}

\end{proof}
\section{Proof of theorem \ref{theo-sharp-dec-eig-tens-I}}\label{sec-theo-sharp-dec-eig-tens-I}
\subsection{Final step in the proof of theorem \ref{theo-sharp-dec-eig-tens-I}}

We first need a general lemma ensuring that both the gradient and the weighted laplacian of some tensor satisfying the elliptic equation in theorem \ref{theo-sharp-dec-eig-tens-I} lie in some exponentially weighted space as soon as this tensor does. More precisely,
\begin{lemma}\label{lemma-a-prio-dec-der-tens}
Let $(M^n,g,\nabla^gf)$ be an expanding gradient Ricci soliton asymptotically conical. Assume $h$ is a tensor satisfying 
\begin{eqnarray*}
&&-\Delta_fh-\lambda h=V_1\ast h+V_2\ast\nabla h,\quad\lambda\in\mathbb{R},\\
&&he^{F_{\alpha}}\in L^2(d\mu_f), \quad\mbox{for some real number $\alpha$},
\end{eqnarray*}
where $V_1$ and $V_2$ are as in theorem \ref{theo-sharp-dec-eig-tens-I}.

Then, $(\nabla h)e^{F_{\alpha-1}}\in L^2(d\mu_f) $ and $(\Delta_fh)e^{F_{\alpha}}\in L^2(d\mu_f)$. 
\end{lemma}

\begin{proof}

Let $\phi$ be a standard smooth cut-off function. Then, on one hand,
\begin{eqnarray*}
\int_M<-\Delta_fh,\phi^2h>e^{2F_{\alpha}}d\mu_f&=&\int_M\phi^2\arrowvert\nabla h\arrowvert^2e^{2F_{\alpha}}d\mu_f+2\int_M<\nabla_{\nabla \phi}h,\phi h>e^{2F_{\alpha}}d\mu_f\\
&&+\frac{1}{2}\int_M<\nabla \arrowvert h\arrowvert^2,\phi^2\nabla e^{2F_{\alpha}}>d\mu_f\\
&\gtrsim&\frac{1}{2}\int_M\phi^2\arrowvert\nabla h\arrowvert^2e^{2F_{\alpha}}d\mu_f\\
&&-\int_M\left(\arrowvert\nabla \phi\arrowvert^2+e^{-2F_{\alpha}}\div_f\left(\phi^2 \nabla e^{2F_{\alpha}}\right)\right)\arrowvert h\arrowvert^2e^{2F_{\alpha}}d\mu_f.
\end{eqnarray*}
On the other hand, 
\begin{eqnarray*}
\int_M<-\Delta_fh,\phi^2h>e^{2F_{\alpha}}d\mu_f\lesssim \int_M\phi^2\arrowvert h\arrowvert^2e^{2F_{\alpha}}d\mu_f+\int_M\phi^2\arrowvert\nabla h\arrowvert^2e^{2F_{\alpha-1}}d\mu_f.
\end{eqnarray*}
These two estimates show then that $(\nabla h) e^{F_{\alpha}}\in L^2(d\mu_f)$ if $he^{F_{\alpha+1}}\in L^2(d\mu_f)$.
As $h$ satisfies an elliptic equation of the form given as above, we conclude by inspecting each term that $(\Delta_fh)e^{F_{\alpha}}\in L^2(d\mu_f)$ if  $he^{F_{\alpha}}\in L^2(d\mu_f)$.

\end{proof}

\begin{rk}
Lemma \ref{lemma-a-prio-dec-der-tens} is far from being optimal : one only needs $V_1=\textit{O}(1)$ and $V_2=\textit{O}(v^{-1/2})$.
\end{rk}
We are in a position to prove theorem \ref{theo-sharp-dec-eig-tens-I} in the case one knows a priori that a solution to equation (\ref{ell-equ-sharp-dec-eig-tens-I}) lies in $L^2(e^{2F_{\alpha}}d\mu_f)$ for any positive $\alpha$.
\begin{theo}\label{theo-fin-step}
Let $(M^n,g,\nabla^gf)$ be an expanding gradient Ricci soliton asymptotically conical satisfying $\HypI$ or $\HypII$. Assume $h$ is a tensor satisfying 
\begin{eqnarray*}
&&-\Delta_fh+\lambda h=V_1\ast h+V_2\ast\nabla h,\quad V_1=\textit{O}(1),\quad V_2=\textit{o}(v^{-1/2}),\quad \lambda\in\mathbb{R},\\
&&he^{F_{\alpha}}\in L^2(d\mu_f), \quad\forall \alpha>0.
\end{eqnarray*}
Then $h\equiv 0$.

\end{theo}
\begin{proof}
Let $K$ be a compact subset of $M$ such that lemma \ref{lemm-cruc-C^0} holds.

Let $\eta$ be a smooth nonnegative function on $M$ such that $\eta\equiv 0$ on $K$ and $\eta\equiv 1$ on $M\setminus K'$ with $K\subset K'$. Thanks to lemma \ref{lemma-a-prio-dec-der-tens}, one can apply lemma \ref{lemm-cruc-C^0} to $\tilde{h}:=\eta h$ for any positive $\alpha$ : 

\begin{eqnarray*}
\alpha I^0_{\alpha}(\tilde{h})+\alpha^2I^0_{\alpha-2}(\tilde{h})\lesssim \left(1+\frac{1}{\alpha}\right) \|[(H-\lambda)\tilde{h}]e^{F_{\alpha}}\|^2_{L^2_f}.
\end{eqnarray*}

Now,
\begin{eqnarray*}
H\tilde{h}-\lambda\tilde{h}&=&\left(hH\eta-2<\nabla h,\nabla\eta>\right)+V_0\ast(\eta h)+V_1\ast \eta\nabla h\\
&=&\eta(h)+V_0\ast(\eta h)+V_1\ast \nabla(\eta h),
\end{eqnarray*}
where $\eta(h)$ is a compactly supported tensor depending on $h$ and $\nabla h$. Therefore,
\begin{eqnarray}\label{est-J-0-alpha}
\|[(H-\lambda)\tilde{h}]e^{F_{\alpha}}\|_{L^2_f}\leq \|[V_0\ast\tilde{h}]e^{F_{\alpha}}\|_{L^2_f}+\|[V_1\ast\nabla\tilde{h}]e^{F_{\alpha}}\|_{L^2_f}+C(K',h)e^{\sup_{K'}F_{\alpha}+f/2}.
\end{eqnarray}
Since $V_0$ is bounded, the first term on the right hand side can be absorbed if $\alpha$ is large enough so that : 
\begin{eqnarray*}
\alpha I^0_{\alpha}(\tilde{h})+\alpha^2I^0_{\alpha-2}(\tilde{h})\lesssim \|[V_1\ast\nabla\tilde{h}]e^{F_{\alpha}}\|^2_{L^2_f}+C(K',h)e^{\sup_{K'}2F_{\alpha}+f}.
\end{eqnarray*}
Now, as $V_1=\textit{O}(v^{-1/2})$, we can bound the remaining term on the right hand side as follows : 
\begin{eqnarray*}
\|[V_1\ast\nabla\tilde{h}]e^{F_{\alpha}}\|^2_{L^2_f}\lesssim\int _M\arrowvert\nabla \tilde{h}\arrowvert^2e^{2F_{\alpha-1}}d\mu_f = I^1_{\alpha-2}(\tilde{h}).
\end{eqnarray*}

By proposition \ref{prop-cov-der-est-wei-nor},
\begin{eqnarray*}
I^1_{\alpha-2}(\tilde{h})\lesssim J^0_{\alpha-1}(\tilde{h})+I^0_{\alpha+1-1}(\tilde{h})+\alpha I^0_{\alpha-1}(\tilde{h})+\alpha^2 I^0_{\alpha-1-1}(\tilde{h}).
\end{eqnarray*}
Similarly to (\ref{est-J-0-alpha}), one has,
\begin{eqnarray*}
J^0_{\alpha-1}(\tilde{h})\lesssim I^0_{\alpha-1}(\tilde{h})+I^1_{\alpha-3}(\tilde{h})+C(K',h)e^{\sup_{K'}2F_{\alpha-1}+f},
\end{eqnarray*}
which implies,
\begin{eqnarray*}
I^1_{\alpha-2}(\tilde{h})\lesssim I^0_{\alpha+1-1}(\tilde{h})+\alpha I^0_{\alpha-1}(\tilde{h})+\alpha^2 I^0_{\alpha-1-1}(\tilde{h})+C(K',h)e^{\sup_{K'}2F_{\alpha-1}+f}.
\end{eqnarray*}
We can conclude by concatenating all the previous inequalities by using now that $V_1=\textit{o}(v^{-1/2})$ :
\begin{eqnarray*}
\alpha I^0_{\alpha}(\tilde{h}) + \alpha^2 I^0_{\alpha-2}(\tilde{h})\lesssim C(K',h)e^{\sup_{K'}2F_{\alpha}+f},
\end{eqnarray*}
which implies in particular that for any positive $\alpha$ : 
\begin{eqnarray*}
\alpha \int_{M\setminus K'}\arrowvert h\arrowvert^2d\mu_f\lesssim C(K',h)e^{\sup_{K'}(2F_{\alpha}+f)-\inf_{M\setminus K'}(2F_{\alpha}+f)}\lesssim C(K',h),
\end{eqnarray*}
if $K'=\{f\leq t_0\}$ for some $t_0$ large enough. Hence $h\equiv 0$ outside a compact set. As $h$ is analytic, since it is a solution of an elliptic equation with analytic coefficients \cite{Ban-Ana}, $h$ vanishes identically on $M$.

\end{proof}

\subsection{A priori faster than polynomial decay}\label{subsec-pol-dec-a-prio-I}

Define $$S:=\{\alpha\in\mathbb{R}\quad|\quad he^{F_{\alpha}}\in L^2(d\mu_f)\}.$$

\begin{theo}\label{theo-sharp-dec-eig-tens}
Let $(M^n,g,\nabla^gf)$ be an asymptotically conical expanding gradient Ricci soliton. Assume $h$ is a smooth tensor satisfying for some $\lambda\in\mathbb{R}$,
\begin{eqnarray*}
&&-\Delta_fh-\lambda h=V_0\ast h+V_1\ast\nabla h,\\
&& \nabla^kV_0=\textit{O}(v^{-1-k/2}), \quad k\in\{0,1,2\},\\
&& \nabla^kV_1=\textit{O}(v^{-3/2-k/2}), \quad k\in\{0,1,2\},\\
&&\lim_{+\infty}v^{-\lambda+\frac{n}{2}}e^vh=\lim_{+\infty}e^{F_0}h=\lim_{+\infty}h_{\lambda}=0.
\end{eqnarray*}
Then,
$$S=\mathbb{R}.$$
\end{theo}

\begin{proof}[Proof of theorem \ref{theo-sharp-dec-eig-tens} ]

We proceed by showing that $S$ is a non empty open and closed set of $\mathbb{R}$.

\begin{claim}\label{S-non-emp}
$S$ is non empty.
\end{claim}
\begin{proof}[Proof of claim \ref{S-non-emp}]

We are in a position to apply theorem \ref{a-prio-point-bd-orns-op} to $h$. Therefore, the two first covariant derivatives of $h_{\lambda}$ decay like the following :
\begin{eqnarray*}
\nabla^kh_{\lambda}=\textit{O}(v^{-k/2}),\quad k=1,2.
\end{eqnarray*}
Now, from the proof of theorem \ref{a-prio-point-bd-orns-op}, $h_{\lambda}$ satisfies the following elliptic equation : 
\begin{eqnarray}\label{eq-back-wei-eigen-ten}
\Delta_{-f}h_{\lambda}&=&W_0\ast h_{\lambda}+W_1\ast\nabla h_{\lambda},
\end{eqnarray}
where $W_0$ behaves like $V_0$ at infinity and $W_1$ is such that
\begin{eqnarray*}
\mathscr{W}_{1,k}:=\sup_Mv^{k/2}\arrowvert\nabla^k (v^{1/2} W_1)\arrowvert<+\infty,\quad k=0,1,2.
\end{eqnarray*}

We conclude by integrating equation (\ref{eq-back-wei-eigen-ten}) along the Morse flow generated by $\nabla f/\arrowvert\nabla f\arrowvert^2$ : 
\begin{eqnarray*}
\partial_t\arrowvert h_{\lambda}\arrowvert^2=2<\nabla_{\frac{\nabla f}{\arrowvert\nabla f\arrowvert^2}}h_{\lambda},h_{\lambda}>=\textit{O}(v^{-2}),
\end{eqnarray*}
i.e. $h_{\lambda}=\textit{O}(v^{-1})$. In particular, $he^{F_{1}}\in L^{2}(d\mu_f)$, i.e. $1\in S$. We even get $(-\infty,1]\subset S$.

\end{proof}

\begin{claim}\label{S-clo}
$S$ is closed.
\end{claim}
\begin{proof}[Proof of claim \ref{S-clo}]
Indeed, let $(\alpha_i)_i$ be a sequence of positive numbers in $S$ converging to $\alpha_{\infty}\in\mathbb{R}_+^*$. Then, by using the proof of theorem \ref{theo-fin-step} and its notations, and by using the assumptions on $V_0$ and $V_1$ at infinity, one has, for every index $i$,
\begin{eqnarray*}
\alpha_i I^0_{\alpha_i}(\tilde{h}) \lesssim_i C(K',h)e^{\sup_{K'}2F_{\alpha_i}+f},
\end{eqnarray*}
where $ \lesssim_i $ means "less than" up to a multiplicative constant independent of $i$ and where $K'$ is a sufficiently large compact subset of $M$ independent of $i$ : as $(\alpha_i)_i$ lies in a compact set, the lower order terms $\alpha_i^2I^0_{\alpha_i-k}(\tilde{h})$, for $k>0$ can be absorbed by $I^0_{\alpha_i}(\tilde{h})$ independently of $i$. Hence, by letting $i$ go to $+\infty$,
 \begin{eqnarray*}
\alpha_{\infty} I^0_{\alpha_{\infty}}(\tilde{h}) \lesssim C(K',h)e^{\sup_{K'}2F_{\alpha_{\infty}}+f}<+\infty,
\end{eqnarray*}
i.e. $\alpha_{\infty}\in S$.

\end{proof}

\begin{claim}\label{S-ope} $S$ is open.
\end{claim}
\begin{proof}[Proof of claim \ref{S-ope}]
Analogously to \cite{Don-Ale-Spe}, define for $\epsilon>0$ and $\alpha\in S$,
\begin{eqnarray}
F_{\alpha,\epsilon,s}:=F_{\alpha}+\frac{\epsilon}{2}\chi_s(\ln v),
\end{eqnarray}
where, if $s$ is positive,
\begin{eqnarray*}
\chi_s(t):=\int_0^t\frac{dx}{1+s^2x^2}=\frac{\arctan(st)}{s}.
\end{eqnarray*}

First, we gather several technical remarks concerning $\chi_s$ in the following lemma :
\begin{lemma}\label{lem-est-ad-hoc-app-id} 
\begin{eqnarray*}
&&0\leq\chi_s(t)\leq c(s) \quad\mbox{on $[0,+\infty)$},\\
 &&\sup_{[0,+\infty)}t^{i-1}\arrowvert\chi_s^{(i)}(t)\arrowvert\leq c,\quad \mbox{for $i=1,2,3$}, \\
&&\lim_{s\rightarrow 0}\chi_s(t)=t,\quad\lim_{s\rightarrow 0}F_{\alpha,\epsilon,s}=F_{\alpha+\epsilon},\\
&&\arrowvert\nabla (F_{\alpha,\epsilon,s}-F_{\alpha})\arrowvert\leq\frac{c\epsilon}{v^{1/2}},\quad \arrowvert\nabla^2 (F_{\alpha,\epsilon,s}-F_{\alpha})\arrowvert\leq \frac{c\epsilon}{v},\\
&&\arrowvert\nabla f\cdot\nabla f\cdot\nabla (w_{\alpha,\epsilon,s}-w_{\alpha})\arrowvert=\frac{\epsilon}{2}\arrowvert\nabla f\cdot\nabla f\cdot\nabla\chi_s(\ln v)\arrowvert\leq c\epsilon.
\end{eqnarray*}
\end{lemma}
 One gets as in the proof of lemma \ref{lemm-cruc-C^0}, for any positive $s$ and $\alpha\in S$ and any tensor $h$ supported outside a sufficiently large compact set $K\subset M$,
\begin{eqnarray*}
&&\int_M<Hh,h>e^{2F_{\alpha,\epsilon,s}}+\left(\arrowvert\nabla F_{\alpha,\epsilon,s}\arrowvert^2+\nabla f\cdot(\arrowvert\nabla F_{\alpha,\epsilon,s}\arrowvert^2-\nabla f\cdot\nabla w_{\alpha,\epsilon,s})-\frac{v}{2}\right)\arrowvert h\arrowvert^2e^{2F_{\alpha,\epsilon,s}}d\mu_f\\
&&\lesssim  \|[(H-\lambda)h]e^{F_{\alpha,\epsilon,s}}\|^2_{L^2_f}\\
&&+\int_M-2\Ric(g)(\nabla T_{\alpha,\epsilon,s},\nabla T_{\alpha,\epsilon,s})+\arrowvert\nabla\Rm(g)\arrowvert\arrowvert\nabla T_{\alpha,\epsilon,s}\arrowvert\arrowvert T_{\alpha,\epsilon,s}\arrowvert +\arrowvert\nabla\Rm(g)\arrowvert\arrowvert\nabla f\arrowvert\arrowvert T_{\alpha,\epsilon,s}\arrowvert^{2} d\mu_f,\\
\end{eqnarray*}
where $T_{\alpha,\epsilon,s}:=he^{F_{\alpha,\epsilon,s}}.$
By the previous estimates on $F_{\alpha,\epsilon,s}$, one has for any positive $s$ and $\alpha\in S$ and any tensor $h$ supported outside a sufficiently large compact set $K\subset M$,
\begin{eqnarray*}
&&\int_M(<Hh,h>-c\epsilon)e^{2F_{\alpha,\epsilon,s}}+\left(\arrowvert\nabla F_{\alpha}\arrowvert^2+\nabla f\cdot(\arrowvert\nabla F_{\alpha}\arrowvert^2-\nabla f\cdot\nabla w_{\alpha})-\frac{v}{2}\right)\arrowvert h\arrowvert^2e^{2F_{\alpha,\epsilon,s}}d\mu_f\\
&&\lesssim  \|[(H-\lambda)h]e^{F_{\alpha,\epsilon,s}}\|^2_{L^2_f}\\
&&+\int_M-2\Ric(g)(\nabla T_{\alpha,\epsilon,s},\nabla T_{\alpha,\epsilon,s})+\arrowvert\nabla\Rm(g)\arrowvert\arrowvert\nabla T_{\alpha,\epsilon,s}\arrowvert\arrowvert T_{\alpha,\epsilon,s}\arrowvert +\arrowvert\nabla\Rm(g)\arrowvert\arrowvert\nabla f\arrowvert\arrowvert T_{\alpha,\epsilon,s}\arrowvert^{2} d\mu_f.\\
\end{eqnarray*}
We define the corresponding quantities $I^k_{\alpha,\epsilon,s}(h)$, $J^k_{\alpha,\epsilon,s}(h)$ depending on $\alpha$, $\epsilon$ and $s$ with obvious notations. Again, by the proof of lemma \ref{lemm-cruc-C^0}, there exists a compact set $K\subset M$ such that for any $\alpha\in S$, for any tensor $h$ compactly supported outside $K$ and any positive $s$, 
\begin{eqnarray}\label{ineq-openness}
(\alpha-\epsilon)I^0_{\alpha,\epsilon,s}(h)+\alpha^2I^0_{\alpha-2,\epsilon,s}(h)\lesssim_s \left(1+\frac{1}{\alpha}\right) \|[(H-\lambda)h]e^{F_{\alpha,\epsilon,s}}\|^2_{L^2_f}.
\end{eqnarray}
Mimicking the proof of theorem \ref{theo-fin-step}, with the same notations, define $\tilde{h}:=\eta h$, where $\eta$ is a smooth nonnegative function such that $\eta\equiv 0$ on $K$ and $\eta\equiv 1$ on $M\setminus K'$ with $K\subset K'$, one gets,
\begin{eqnarray*}
\|[(H-\lambda)\tilde{h}]e^{F_{\alpha,\epsilon,s}}\|_{L^2_f}\leq \|[V_0\ast\tilde{h}]e^{F_{\alpha,\epsilon,s}}\|_{L^2_f}+\|[V_1\ast\nabla\tilde{h}]e^{F_{\alpha,\epsilon,s}}\|_{L^2_f}+C(K',h)e^{\sup_{K'}F_{\alpha,\epsilon,s}+f/2}.
\end{eqnarray*}
As $V_0=\textit{o}(1)$ at infinity, $K$ can be chosen sufficiently large such that $\sup_{M\setminus K}\arrowvert V_0\arrowvert\leq\epsilon$. Therefore, the term $\|[V_0\ast\tilde{h}]e^{F_{\alpha,\epsilon,s}}\|_{L^2_f}$ can be absorbed by the left hand side of inequality (\ref{ineq-openness}). It suffices to handle the term involving the first covariant derivative of $\tilde{h}$. To do so, we proceed as in the proof of proposition \ref{prop-cov-der-est-wei-nor}, and we estimate as follows
\begin{eqnarray*}
\int_M\arrowvert\nabla \tilde{h}\arrowvert^2e^{2F_{\alpha-1,\epsilon,s}}d\mu_f&\lesssim_s&J^0_{\alpha-1,\epsilon,s}(\tilde{h})+I^0_{\alpha,\epsilon,s}(\tilde{h})+(\alpha+\epsilon) I^0_{\alpha-1,\epsilon,s}(\tilde{h})+(\alpha+\epsilon)^2 I^0_{\alpha-2,\epsilon,s}(\tilde{h}).
\end{eqnarray*}
As $V_1=\textit{o}(v^{-1/2})$, one can assume that $\sup_{M\setminus K}\arrowvert v^{1/2}V_1\arrowvert\leq \epsilon.$ Therefore, if $\epsilon$ is chosen small enough (depending on $\alpha$),
\begin{eqnarray*}
(\alpha-2\epsilon)I^0_{\alpha,\epsilon,s}\lesssim_s C(K',h)e^{\sup_{K'}(2F_{\alpha,\epsilon,s}+f)},
\end{eqnarray*}
for any positive $s$ which implies in particular that $he^{F_{\alpha+\epsilon}}\in L^2(d\mu_f)$, i.e. $S$ is open.

\end{proof}

\end{proof}

\section{Proof of theorem \ref{Main-theo-uni-inf-0}}\label{section-proof-main-theo}

We proceed as in the proof of theorem \ref{theo-sharp-dec-eig-tens-I}. Let's make some remarks first to simplify the analysis. Let $(M_i^n,g_i,\nabla^{g_i}f_i)_{i=0,1}$ be two normalized expanding gradient Ricci solitons that are asymptotically conical with isometric asymptotic cones $(C(X_i),dr^2+r^2g_{X_i},r\partial_r/2)_{i=0,1}$. Let $(\phi_{i,t})_{i=0,1}$ be the flows generated by the vector fields $(\nabla^{g_i}f_i/\arrowvert\nabla^{g_i}f_i\arrowvert^2)_{i=0,1}$. Then, by \cite{Der-Asy-Com-Egs}, one defines the following diffeomorphisms at infinity for $i=0,1$ :
\begin{eqnarray*}
\phi_i:(t_0,+\infty)\times X_i &\rightarrow& \{f_i+\mu(g_i)> t_0^2/4\}\\
(t,x)&\rightarrow& \phi_{i,\frac{t^2}{4}-\frac{t_0^2}{4}}(x),
\end{eqnarray*}
 where $X_i:= \{f_i+\mu(g_i)=t_0^2/4\}$, for $t_0$ large enough. These diffeomorphisms preserves the expanding structure in the sense of definition \ref{defn-asy-con-egs}. By pulling back tensors defined outside $ \{f_i+\mu(g_i)> t_0^2/4\}$ via the diffeomorphisms $\phi_i$, and up to another pull-back coming from the fact that the asymptotic cones are isometric, one is reduced to consider the following setting with some slight abuse of notations : let $(M^n\setminus K,g_i,\nabla^{g_i}f)_{i=0,1}$, where $K\subset M$ is compact, be two (incomplete) expanding  gradient Ricci solitons with the same asymptotic cone, same potential function and vanishing entropy, i.e. the following soliton identities hold : 
 \begin{eqnarray}
&&\Delta_{g_i}f=R_{g_i}+\frac{n}{2},\label{poisson-egs}\\
&&\arrowvert\nabla^{g_i}f\arrowvert^2+\R_{g_i}=f,\label{first-ham-egs}\\
&&\nabla^{g_i}\R_{g_i}+2\Ric(g_i)(\nabla^{g_i}f)=0.\label{Bianchi-egs}
\end{eqnarray}

 Now, as explained in the introduction, the metrics $(g_i)_{i=0,1}$ do not satisfy a nice strictly elliptic equation due to the invariance of the equation of an expanding gradient Ricci soliton under the action of the diffeomorphisms. Nonetheless, the Ricci tensor of an expander does satisfy a nice elliptic equation, more precisely, we have, in this setting, for $i=0,1$,
 \begin{eqnarray}\label{eq-ric-pull-back}
\Delta_{g_i,f}\Ric(g_i)+2\Rm(g_i)\ast\Ric(g_i)=-\Ric(g_i), \quad\mbox{on $M\setminus K$.}
\end{eqnarray}

Now, we consider the difference of the Ricci tensors $h:=\Ric(g_1)-\Ric(g_0)$ and the difference of the metrics  $H:=g_1-g_0$ and compute the evolution equation satisfied by $h$ with respect to the metric $g_0$ denoted once and for all by $g$. 

\begin{prop}\label{prop-mise-en-equ}
With the notations as above,
\begin{eqnarray}
\Delta_{g,f}h+h+\Rm(g)\ast h&=& \nabla^{g,2}H\ast W_0+ \nabla^{g}H\ast W_1+H\ast W_2,\label{eq-ric-I}\\
h\ast V_0&=&\nabla_{\nabla^gf}H+H\ast V_1,\label{eq-ric-II}
\end{eqnarray}
where $(W_i)_{i=0,1,2}$, $(V_i)_{i=1,2}$ are tensors made out of the potential function $f$ and the curvatures of $g$ and $g_1$ satisfying 
\begin{eqnarray*}
&&\nabla^{g,k}W_0=\textit{O}(v^{-1-k/2}),\quad \nabla^{g,k}W_1=\textit{O}(v^{-1/2-k/2}),\quad \nabla^{g,k}W_2=\textit{O}(v^{-1-k/2}),\\&&\nabla^{g,k}V_0=\textit{O}(v^{-k/2}),\quad \nabla^{g,k}V_1=\textit{O}(v^{-1-k/2}),
\end{eqnarray*}
for any nonnegative integer $k$.

\end{prop}

\begin{proof}
Using (\ref{eq-ric-pull-back}),
\begin{eqnarray*}
\Delta_{g,f}h+h+\Rm(g)\ast h&=&\Delta_{g,f}\Ric(g_1)+\Ric(g_1)+\Rm(g)\ast \Ric(g_1)\\
&=&-\left[\left(\Delta_{g_1}-\Delta_g\right)+\left(\nabla^{g_1}_{\nabla^{g_1}f}-\nabla^{g}_{\nabla^{g}f}\right)+\left(\Rm(g_1)-\Rm(g)\right)\ast\right]\Ric(g_1).
\end{eqnarray*}

Now, if $T$ is a tensor,
\begin{eqnarray*}
\nabla^{g_1}T&=&\nabla^gT+g_1^{-1}\ast\nabla^g(g_1-g)\ast T=\nabla^gT+g_1^{-1}\ast\nabla^gH\ast T,\\
\nabla^{g_1,2}T&=&\nabla^{g,2}T+g_1^{-1}\ast\nabla^{g,2}H\ast T+ g_1^{-1}\ast\nabla^{g}H\ast \nabla^gT+g_1^{-2}\ast\nabla^{g}H^{\ast 2}\ast T,\\
\Rm(g_1)&=&\Rm(g)+g_1^{-1}\ast\nabla^{g,2}H+g_1^{-2}\ast\nabla^{g}H^{\ast 2},\\
\nabla^{g_1}_{\nabla^{g_1}f}T&=&\nabla^{g}_{\nabla^{g}f}T+g_1^{-1}\ast H\ast\nabla^g f\ast\nabla^g T+g_1^{-1}\ast\nabla^gH\ast\nabla^gf\ast T.
\end{eqnarray*}
Therefore, after dropping the contractions with $g_1^{-1}$,
\begin{eqnarray*}
\Delta_{g,f}h+h+\Rm(g)\ast h&=& \nabla^{g,2}H\ast \Ric(g_1)\\
&&+ \nabla^{g}H\ast (\nabla^g\Ric(g_1)+\nabla^gf\ast \Ric(g_1)+\nabla^{g}H\ast \Ric(g_1))\\
&&+ H\ast(\nabla^g f\ast\nabla^g \Ric(g_1)+\nabla^{g,2}\Ric(g_1)).\\
\end{eqnarray*}

Finally, we link the two quantities $H$ and $h$ with the help of the soliton equation together with the soliton identities given by lemma \ref{id-EGS}. 
\begin{eqnarray*}
2h=2(\Ric(g_1)-\Ric(g))&=&\Li_{\nabla^{g_1}f}(g_1)-\Li_{\nabla^gf}(g)-(g_1-g)\\
&=&\Li_{\nabla^gf}(g_1-g)+\Li_{\nabla^{g_1}f-\nabla^gf}(g_1)-(g_1-g)\\
&=&\Li_{\nabla^gf}(H)+\Li_{\nabla^{g_1}f-\nabla^gf}(g_1)-H\\
&=&\nabla^g_{\nabla^gf}H+\Ric(g)\ast H+\Li_{\nabla^{g_1}f-\nabla^gf}(g_1).
\end{eqnarray*}
It turns out that the rough estimate $\nabla^{g_1}f-\nabla^gf=H\ast\nabla^gf$ is not enough. Actually, by using the fact that the diffeomorphisms preserve the expanding structure,
\begin{eqnarray*}
\nabla^{g_1}f-\nabla^gf&=&(g_1^{tt}-g^{tt})\left(\sqrt{f}\partial_t\right),\\
(g_1)(\partial_t,\partial_t)&=&f\cdot g_1\left(\frac{\nabla^{g_1}f}{\arrowvert\nabla^{g_1} f\arrowvert^2},\frac{\nabla^{g_1}f}{\arrowvert\nabla^{g_1} f\arrowvert^2}\right)=\frac{f}{\arrowvert\nabla^{g_1} f\arrowvert^2}=\frac{1}{1-f^{-1}\R_{g_1}},\\
g(\partial_t,\partial_t)&=&\frac{1}{1-f^{-1}\R_{g}},\\
\end{eqnarray*}
where $t$ denotes the radial coordinate.
This implies, by using extensively the soliton identities (\ref{poisson-egs}), (\ref{first-ham-egs}), (\ref{Bianchi-egs}),
\begin{eqnarray*}
\nabla^{g_1}f-\nabla^gf&=&(\R_g-\R_{g_1})f^{-1/2}\partial_t,\\
\nabla^{g_1}(\nabla^{g_1}f-\nabla^gf)&=&\nabla^{g_1}(\R_g-\R_{g_1})\ast (f^{-1/2}\partial_t)+(\R_g-\R_{g_1})\nabla^{g_1}(f^{-1/2}\partial_t)\\
&=&\left[2\Ric(g_1)(\nabla^{g_1}f)-2\Ric(g)(\nabla^gf)+H\ast\nabla^g\R_g\right]\ast f^{-1/2}\partial_t\\
&&+[\tr_gh+H\ast\Ric(g_1)]\ast\nabla^{g_1}(f^{-1/2}\partial_t)\\
&=&h\ast\nabla^{g}f\ast f^{-1/2}\partial_t+\tr_gh\ast \nabla^{g_1}(f^{-1/2}\partial_t)\\
&&+H\ast\left([\Ric(g_1)\ast\nabla^g f+\nabla^g\R_g]\ast (f^{-1/2}\partial_t)+\Ric(g_1)\ast\nabla^{g_1}(f^{-1/2}\partial_t)\right).
\end{eqnarray*}
Therefore,
\begin{eqnarray*}
h\ast V_0=\nabla^g_{\nabla^gf}H+H\ast V_1,
\end{eqnarray*}
where $V_0$ is a tensor satisfying $\nabla^{g,k}V_0=\textit{O}(v^{-k/2})$, for any nonnegative integer $k$ and $V_1$ behaves at infinity like the Ricci curvature of $g$ or $g_1$, that is $\nabla^{g,k}V_1=\textit{O}(v^{-1-k/2})$, for any nonnegative integer $k$.

\end{proof}

\subsection{Final step of the proof of theorem \ref{Main-theo-uni-inf-0}}

As in subsection \ref{subsec-pol-dec-a-prio-I}, we prove theorem \ref{Main-theo-uni-inf-0} in case we know a priori that the difference $h:=\Ric(g_1)-\Ric(g)$ decays faster than polynomially at infinity. \begin{theo}\label{theo-fin-step-II}
With the above notations, assume that $he^{F_{\alpha}}\in L^2(d\mu_f)$ and $\left(\nabla^{g,k}H\right)e^{F_{\alpha+2-k}}\in L^2(d\mu_f)$ for any $\alpha >0$ and $k=0,1,2$, then $h\equiv 0$ outside a compact set. In particular, it implies $H=g_1-g=0$ outside a compact set.
\end{theo}

\begin{rk}

The condition on the covariant derivatives of $H$ is consistent with proposition \ref{prop-control-rad-der} and equation (\ref{eq-ric-II}).
\end{rk}

\begin{proof}
Let $K$ be a compact subset of $M$ such that lemma \ref{lemm-cruc-C^0} holds for tensors supported outside $K$. By equation (\ref{eq-ric-I}), we also know that $(\Delta_fh)e^{F_{\alpha}}\in L^2(d\mu_f)$ for any positive $\alpha$ and lemma \ref{prop-cov-der-est-wei-nor} implies $(\nabla^{g,k}h)e^{F_{\alpha-k}}\in L^2(d\mu_f)$ for any $k=1,2$ and any positive $\alpha$.

Let $\eta$ be a smooth nonnegative function on $M$ such that $\eta\equiv 0$ on $K$ and $\eta\equiv 1$ on $M\setminus K'$ with $K\subset K'$. Apply lemma \ref{lemm-cruc-C^0} to $\tilde{h}:=\eta h$ : 
\begin{eqnarray}
\alpha I^0_{\alpha}(\tilde{h})+\alpha^2 I^0_{\alpha-2}(\tilde{h})\l\lesssim \|[(\Delta_{g,f}+1)\tilde{h}]e^{F_{\alpha}}\|^2_{L^2_f}.\label{eq-amorce}
\end{eqnarray}

On the other hand, $\tilde{h}$ and $\tilde{H}:=\eta H$ satisfy the following equations similar to equations (\ref{eq-ric-I}) and (\ref{eq-ric-II}) : 
\begin{eqnarray}
\Delta_{g,f}\tilde{h}+\tilde{h}+2\Rm(g)\ast \tilde{h}&=& \nabla^{g,2}\tilde{H}\ast W_0+ \nabla^{g}\tilde{H}\ast W_1+\tilde{H}\ast W_2+T(h,H,\eta)\label{eq-ric-I-bis}\\
\tilde{h}\ast V_0&=&\nabla_{\nabla^gf}\tilde{H}+\tilde{H}\ast V_1+T(H,\eta),\label{eq-ric-II-bis}
\end{eqnarray}
 where $(W_k)_{k=0,1,2}$ and $(V_i)_{i=1,2}$ are tensors decaying at infinity as in equations (\ref{eq-ric-I}) and (\ref{eq-ric-II}) and where $T(h,H,\eta)$ and $T(H,\eta)$ are compactly supported tensors. 
For $\alpha$ large enough so that propositions \ref{prop-control-rad-der} and \ref{prop-cov-der-est-wei-nor} are applicable, we estimate the righthand side of (\ref{eq-amorce}) as follows using freely equations (\ref{eq-ric-I-bis}) and (\ref{eq-ric-II-bis}) : 
\begin{eqnarray*}
\|[(\Delta_{g,f}+1)\tilde{h}]e^{F_{\alpha}}\|^2_{L^2_f}&\lesssim&e^{\sup_{K'}(2F_{\alpha}+f)}+\sum_{k=0}^2\|[\nabla^{g,2-k}\tilde{H}\ast W_k]e^{F_{\alpha}}\|^2_{L^2_f}+I_{\alpha-1}^0(\tilde{h}).
\end{eqnarray*}
Now,
\begin{eqnarray*}
\|[\tilde{H}\ast W_2]e^{F_{\alpha}}\|^2_{L^2_f}&\lesssim& I^0_{\alpha}(\tilde{H})\l\lesssim I^0_{\alpha-2}(\nabla^g_{\nabla^g f}\tilde{H})\\
&\lesssim&I^0_{\alpha-2}(\tilde{h})+e^{\sup_{K'}(2F_{\alpha}+f)},\\
\|[\nabla^g\tilde{H}\ast W_1]e^{F_{\alpha}}\|^2_{L^2_f}&\lesssim&\int_M\arrowvert\nabla^g\tilde{H}\arrowvert^2e^{2F_{\alpha}}d\mu_f\\
&\lesssim&I^1_{\alpha-1}(\tilde{H})\lesssim\alpha^{-1}(I^1_{\alpha-2}(\tilde{h})+I^0_{\alpha-2}(\tilde{h}))+e^{\sup_{K'}(2F_{\alpha}+f)}\\
&\lesssim&\alpha^{-1}\left[J^0_{\alpha-1}(\tilde{h})+I^0_{\alpha}(\tilde{h})+\alpha I^0_{\alpha-1}(\tilde{h})+\alpha^2 I^0_{\alpha-2}(\tilde{h})\right]+e^{\sup_{K'}(2F_{\alpha}+f)}\\
&\lesssim&\|[(\Delta_{g,f}+1)\tilde{h}]e^{F_{\alpha-1}}\|^2_{L^2_f}+I^0_{\alpha}(\tilde{h})+\alpha I^0_{\alpha-1}(\tilde{h})+e^{\sup_{K'}(2F_{\alpha}+f)},\\
\|[\nabla^{g,2}\tilde{H}\ast W_0]e^{F_{\alpha}}\|^2_{L^2_f}&\lesssim&\int_M\arrowvert\nabla^{g,2}\tilde{H}\arrowvert^2e^{2F_{\alpha-2}}d\mu_f=I^2_{\alpha-4}(\tilde{H})\\
&\lesssim&\alpha^{-2}\left(\sum_{i=0}^2I^i_{\alpha-4}(\tilde{h})\right)+e^{\sup_{K'}(2F_{\alpha}+f)}\\
&\lesssim&\alpha^{-2}\left(J^0_{\alpha-1}(\tilde{h})+\alpha J^0_{\alpha-2}(\tilde{h})+\alpha^2 J^0_{\alpha-3}(\tilde{h})\right)+\sum_{\arrowvert i\arrowvert\leq 2}\alpha^{-i}I^0_{\alpha-2+i}(\tilde{h})\\
&&+\alpha^{-2}(J^0_{\alpha-3}(\tilde{h})+I^0_{\alpha-2}(\tilde{h})+\alpha I^0_{\alpha-3}(\tilde{h})+\alpha^2 I^0_{\alpha-4}(\tilde{h}))\\
&&+\alpha^{-2}I^0_{\alpha-4}(\tilde{h})+e^{\sup_{K'}(2F_{\alpha}+f)}\\
&\lesssim&J^0_{\alpha-1}(\tilde{h})+I^0_{\alpha}(\tilde{h})+\alpha^2I^0_{\alpha-3}(\tilde{h})+e^{\sup_{K'}(2F_{\alpha}+f)}\\
&\lesssim&\|[(\Delta_{g,f}+1)\tilde{h}]e^{F_{\alpha-1}}\|^2_{L^2_f}+I^0_{\alpha}(\tilde{h})+\alpha^2I^0_{\alpha-3}(\tilde{h})+e^{\sup_{K'}(2F_{\alpha}+f)}.
\end{eqnarray*}
Therefore, by the previous estimates, one gets :
\begin{eqnarray*}
\|[(\Delta_{g,f}+1)\tilde{h}]e^{F_{\alpha}}\|^2_{L^2_f}&\lesssim&\|[(\Delta_{g,f}+1)\tilde{h}]e^{F_{\alpha-1}}\|^2_{L^2_f}+I^0_{\alpha}(\tilde{h})+\alpha I^0_{\alpha-1}(\tilde{h})+\alpha^2I^0_{\alpha-3}(\tilde{h})+e^{\sup_{K'}(2F_{\alpha}+f)},
\end{eqnarray*}
that is,
\begin{eqnarray}
\|[(\Delta_{g,f}+1)\tilde{h}]e^{F_{\alpha}}\|^2_{L^2_f}&\lesssim&I^0_{\alpha}(\tilde{h})+\alpha I^0_{\alpha-1}(\tilde{h})+\alpha^2I^0_{\alpha-3}(\tilde{h})+e^{\sup_{K'}(2F_{\alpha}+f)},\label{eq-amorce-ii}
\end{eqnarray}
Inequalities (\ref{eq-amorce}) and (\ref{eq-amorce-ii}) show that, for any $\alpha$ large enough,
\begin{eqnarray*}
\alpha I^0_{\alpha}(\tilde{h})+\alpha^2I^0_{\alpha-2}(\tilde{h})\lesssim e^{\sup_{K'}(2F_{\alpha}+f)},
\end{eqnarray*}
for some compact $K'\subset M$ independent of $\alpha$. By the very definition of $I^0_{\alpha}(\tilde{h})$, one has
\begin{eqnarray*}
\alpha\int_M\arrowvert\tilde{h}\arrowvert^2d\mu_f\leq C(K'),
\end{eqnarray*}
where $C(K')$ is a positive constant independent of $\alpha$. Therefore, $\tilde{h}\equiv 0$, i.e. $h\equiv 0$ outside a compact set. Using equation (\ref{eq-ric-II}) together with the fact that $\lim_{+\infty}H=0$ give $H\equiv 0$ outside a compact set.

\end{proof}

\subsection{A priori faster than polynomial decay}\label{subsec-pol-dec-a-prio-II}

With the notations of section \ref{section-proof-main-theo}, define analogously to \ref{subsec-pol-dec-a-prio-I}, 
\begin{eqnarray*}
S:=\{\alpha\in\mathbb{R}\quad|\quad he^{F_{\alpha}}\in L^2(d\mu_f),\quad \left(\nabla^{g,k}H\right)e^{F_{\alpha+2-k}}\in L^2(d\mu_f), \quad k=0,1,2\}.
\end{eqnarray*}

\begin{theo}\label{theo-main-S}
 With the notations of section \ref{section-proof-main-theo}, $$S=\mathbb{R}.$$
\end{theo}

\begin{proof}[Proof of theorem \ref{theo-main-S}]
\begin{claim}\label{claim-non-empty-ii}
$S$ is non empty.
\end{claim}

\begin{proof}[Proof of claim \ref{claim-non-empty-ii}]
First of all, we evaluate how fast the tensor $H$ together with its covariant derivatives decay at infinity.

By assumption on the obstruction tensor, i.e. $h=\textit{O}(v^{1-n/2}e^{-v})$, together with the fact that the asymptotic cones are isometric, integrating equation (\ref{eq-ric-II}) along the Morse flow generated by the potential function $f$ leads to 
\begin{eqnarray*}
H=\textit{O}(v^{-n/2}e^{-v}).
\end{eqnarray*}
As the obstruction tensor is smooth, i.e. the rescaled covariant derivative $v^{k/2}\nabla^{g,k}(v^{n/2-1}e^v h)$ is bounded for any nonnegative integer $k$, we have the following estimates by using commutation formulae together with equation (\ref{eq-ric-II}) : 
\begin{eqnarray*}
v^{k/2}\nabla^{g,k}\left(v^{n/2}e^vH\right)=\textit{O}(1),
\end{eqnarray*}
for any nonnegative integer $k\geq 0$. 

Now, going back to equation (\ref{eq-ric-I}), one has
\begin{eqnarray*}
&&\Delta_{g,f}h+h+2\Rm(g)\ast h=Q,\\
\end{eqnarray*}
where $Q$ satisfies 
\begin{eqnarray*}
v(v^{n/2-1}e^vQ)&=&v^{n/2}e^vQ=(v^{n/2}e^vH)\ast W_2+(v^{n/2}e^v\nabla^gH)\ast W_1+(v^{n/2}e^v\nabla^{g,2}H)\ast W_0\\
&=&(v^{n/2}e^vH)\ast \bar{W}_2+\nabla^g(v^{n/2}e^vH)\ast \bar{W}_1+\nabla^{g,2}(v^{n/2}e^vH)\ast \bar{W}_0,
\end{eqnarray*}
where $(\bar{W}_{i})_{i=0,1,2}$ are smooth tensors satisfying 
\begin{eqnarray*}
v^{k/2}\nabla^{g,k}\bar{W}_2=\textit{O}(1),\quad v^{k/2}\nabla^{g,k}\bar{W}_1=\textit{O}(v^{-1/2}),\quad v^{k/2}\nabla^{g,k}\bar{W}_0=\textit{O}(v^{-1}),
\end{eqnarray*}
for any nonnegative integer $k\geq 0$. Recall now that, by equation (\ref{ell-eq-resc-sol}) established in the proof of theorem \ref{a-prio-point-bd-orns-op} that $v^{n/2-1}e^vh=:h_1$ satisfies
\begin{eqnarray}\label{eq-ell-resc-ric-curv-diff}
\Delta_{g,-f}h_1=&\tilde{W}_0\ast h_{1}+\tilde{W}_1\ast\nabla h_{1}+Q_{1},
\end{eqnarray}
where $Q_1:=v^{n/2-1}e^vQ$ and $(\tilde{W}_i)_{i=0,1}$ satisfy
\begin{eqnarray*}
v^{k/2}\nabla^{g,k}(v\tilde{W}_0)=\textit{O}(1),\quad v^{k/2}\nabla^{g,k} (v^{1/2} \tilde{W}_1)=\textit{O}(1),\quad\forall k\geq 0.
\end{eqnarray*}
Finally, by integrating equation (\ref{eq-ell-resc-ric-curv-diff}) along the Morse flow generated by $f$, one has :
\begin{eqnarray*}
h_1=\textit{O}(v^{-1}).
\end{eqnarray*}
In particular, $he^{F_{\alpha}}\in L^2(d\mu_f)$ for any $\alpha<2$. Integrating equation (\ref{eq-ric-II}) again along the  Morse flow generated by $f$ gives 
\begin{eqnarray*}
v^{n/2}e^vH=\textit{O}(v^{-1}),
\end{eqnarray*}
 which implies $He^{F_{\alpha+2}}\in L^2(d\mu_f)$ for any $\alpha<2$. Now, by standard interpolation inequalities applied to $H$, one has $\nabla H=\textit{O}(v^{1/2-\epsilon-n/2}e^{-v})$ and $\nabla^{g,2}H=\textit{O}(v^{1-\epsilon-n/2}e^{-v})$, for some positive $\epsilon$ which implies the result for some positive $\alpha$ sufficiently small compared to $\epsilon$.

\end{proof}

\begin{claim}\label{claim-non-empty-ii}
Define $$\alpha_{+}:=\sup\{\alpha\in\mathbb{R}|\forall \beta\in(-\infty,\alpha),\quad he^{F_{\beta}}\in L^2(d\mu_f),\quad  \left(\nabla^{g,k}H\right)e^{F_{\alpha+2-k}}\in L^2(d\mu_f), \quad k=0,1,2\}.$$ Then $$\alpha_{+}=+\infty.$$
\end{claim}

\begin{proof}[Proof of claim \ref{claim-non-empty-ii}]
Assume on the contrary that $\alpha_{+}<+\infty$. Then $\alpha_+\in S$. 

Indeed, it amounts to show that $S$ is closed. Let $(\alpha_i)_i$ be a sequence of positive number in $S$ converging to $\alpha_+$. The proof consists in adapting the arguments of the proof of theorem \ref{theo-fin-step} : the only difference is that we do not need to take into account the dependence on the multiplicative constants involving the parameter $\alpha$. Therefore, the symbol $\lesssim_i$ will mean up to a multiplicative constant uniform in the indices $i$. Starting from equations (\ref{eq-ric-I-bis}) and (\ref{eq-ric-II-bis}) with the same notations, we get for any index $i$, analogously to inequality (\ref{eq-amorce}),
\begin{eqnarray}\label{eq-amorce-bis}
 I^0_{\alpha_i}(\tilde{h})\lesssim_i \|[(\Delta_{g,f}+1)\tilde{h}]e^{F_{\alpha_i}}\|^2_{L^2_f}.
\end{eqnarray}
Now, we proceed to estimate from above the righthand side of the previous inequality thanks to equations (\ref{eq-ric-I-bis}) and (\ref{eq-ric-II-bis}) : 
\begin{eqnarray*}
\|[(\Delta_{g,f}+1)\tilde{h}]e^{F_{\alpha_i}}\|^2_{L^2_f}&\lesssim_i&C+\sum_{k=0}^2\|[\nabla^{g,2-k}\tilde{H}\ast W_k]e^{F_{\alpha_i}}\|^2_{L^2_f}+I_{\alpha_i-1}^0(\tilde{h}).
\end{eqnarray*}
Now, by assumption $\HypI$, $\limsup_{+\infty}f\arrowvert W_0\arrowvert\leq\epsilon$, for some $\epsilon>0$ sufficiently small. Therefore, there exists a compact set $K_{\epsilon}$ of $M$ sufficiently large but independent of $i$ such that
\begin{eqnarray*}
\|[\tilde{H}\ast W_2]e^{F_{\alpha_i}}\|^2_{L^2_f}&\lesssim_i& I^0_{\alpha_i}(\tilde{H})\l\lesssim_i I^0_{\alpha_i-2}(\nabla^g_{\nabla^g f}\tilde{H})\\
&\lesssim_i&I^0_{\alpha_i-2}(\tilde{h})+C(K_{\epsilon}),\\
\|[\nabla^g\tilde{H}\ast W_1]e^{F_{\alpha_i}}\|^2_{L^2_f}&\lesssim_i&\int_M\arrowvert\nabla^g\tilde{H}\arrowvert^2e^{2F_{\alpha_i-1}}d\mu_f\\
&\lesssim_i&I^1_{\alpha_i-2}(\tilde{H})\lesssim_i I^1_{\alpha_i-3}(\tilde{h})+C(K_{\epsilon})\\
&\lesssim_i&J^0_{\alpha_i-2}(\tilde{h})+I^0_{\alpha_i-1}(\tilde{h})+C(K_{\epsilon})\\
&\lesssim_i&\|[(\Delta_{g,f}+1)\tilde{h}]e^{F_{\alpha_i-2}}\|^2_{L^2_f}+I^0_{\alpha_i-1}(\tilde{h})+C(K_{\epsilon}),\\
\|[\nabla^{g,2}\tilde{H}\ast W_0]e^{F_{\alpha_i}}\|^2_{L^2_f}&\lesssim_i&\epsilon\int_M\arrowvert\nabla^{g,2}\tilde{H}\arrowvert^2e^{2F_{\alpha_i-2}}d\mu_f=\epsilon I^2_{\alpha_i-4}(\tilde{H})\\
&\lesssim_i&\epsilon\sum_{j=0}^2I^j_{\alpha_i-4}(\tilde{h})+C(K_{\epsilon})\\
&\lesssim_i&J^0_{\alpha_i-1}(\tilde{h})+\epsilon I^0_{\alpha_i}(\tilde{h})+C(K_{\epsilon})\\
&\lesssim_i&J^0_{\alpha_i-1}(\tilde{h})+\epsilon I^0_{\alpha_i}(\tilde{h})+C(K_{\epsilon})\\
&\lesssim_i&\|[(\Delta_{g,f}+1)\tilde{h}]e^{F_{\alpha_i-1}}\|^2_{L^2_f}+\epsilon I^0_{\alpha_i}(\tilde{h})+C(K_{\epsilon}).
\end{eqnarray*}

Therefore, by the previous estimates, one gets :
\begin{eqnarray*}
\|[(\Delta_{g,f}+1)\tilde{h}]e^{F_{\alpha_i}}\|^2_{L^2_f}&\lesssim_i&\|[(\Delta_{g,f}+1)\tilde{h}]e^{F_{\alpha_i-1}}\|^2_{L^2_f}+\epsilon I^0_{\alpha_i}(\tilde{h})+C(K_{\epsilon}),
\end{eqnarray*}
that is,
\begin{eqnarray}
\|[(\Delta_{g,f}+1)\tilde{h}]e^{F_{\alpha_i}}\|^2_{L^2_f}&\lesssim_i&\epsilon I^0_{\alpha_i}(\tilde{h})+C(K_{\epsilon}).\label{eq-amorce-ii-bis}
\end{eqnarray}

Inequalities (\ref{eq-amorce-bis}) and (\ref{eq-amorce-ii-bis}) show that, for any positive $\epsilon$ sufficiently small, there exists a compact subset $K_{\epsilon}$ such that for any $\alpha_i$,
\begin{eqnarray*}
 I^0_{\alpha_i}(\tilde{h})\lesssim_i C(K_{\epsilon}).
\end{eqnarray*}
 Hence, $he^{F_{\alpha_+}}\in L^2(d\mu_f)$. Now, by (\ref{eq-amorce-ii-bis}), $(\Delta_fh)e^{F_{\alpha_+}}\in L^2(d\mu_f)$ hence $(\nabla^{g,k}h)e^{F_{\alpha_+-k}}\in L^2(d\mu_f)$ for $k=1,2$ by proposition \ref{prop-cov-der-est-wei-nor}. Finally, proposition \ref{prop-control-rad-der} implies that $(\nabla^{g,k}H)e^{F_{\alpha_++2-k}}\in L^2(d\mu_f)$ for $k=0,1,2$. Hence $\alpha_+\in S$.

We reach a contradiction by showing that there exists some positive $\epsilon$ such that $\alpha_++\epsilon\in S$.

To do so, we proceed as in the proof of claim \ref{S-ope}. With the same notations used in the proof of claim \ref{S-ope}, there exists a compact set $K\subset M$ such that for any $\alpha\in S$, for any tensor $T$ compactly supported outside $K$, any positive $s$ and any positive $\epsilon$ small compared to $\alpha$, 
\begin{eqnarray*}
(\alpha-\epsilon)I^0_{\alpha,\epsilon,s}(T)\lesssim_s \|[(\Delta_{g,f}+1)T]e^{F_{\alpha,\epsilon,s}}\|^2_{L^2_f},
\end{eqnarray*}
where $\lesssim_s$ is uniform in $s$. With the same notations used previously in the proof of claim \ref{claim-non-empty-ii}, we substitute $T$ to $\tilde{h}$ satisfying equations (\ref{eq-ric-I-bis}), (\ref{eq-ric-II-bis}) so that
\begin{eqnarray}\label{ineq-openness-bis}
(\alpha-\epsilon)I^0_{\alpha,\epsilon,s}(\tilde{h})\lesssim_s \|[(\Delta_{g,f}+1)\tilde{h}]e^{F_{\alpha,\epsilon,s}}\|^2_{L^2_f}.
\end{eqnarray}
As before, estimating each term present on the righthand side of equation (\ref{eq-ric-II-bis}) as we already did previously, leads to
\begin{eqnarray*}
\|[\tilde{H}\ast W_2]e^{F_{\alpha,\epsilon,s}}\|^2_{L^2_f}&\lesssim_s&I^0_{\alpha-2,\epsilon,s}(\tilde{h})+C(K_{\alpha}),\\
\|[\nabla^g\tilde{H}\ast W_1]e^{F_{\alpha-1,\epsilon,s}}\|^2_{L^2_f}&\lesssim_s&\|[(\Delta_{g,f}+1)\tilde{h}]e^{F_{\alpha-2,\epsilon,s}}\|^2_{L^2_f}+I^0_{\alpha-1,\epsilon,s}(\tilde{h})+C(K_{\alpha}),\\
\|[\nabla^{g,2}\tilde{H}\ast W_0]e^{F_{\alpha,\epsilon,s}}\|^2_{L^2_f}&\lesssim_s&\epsilon I^2_{\alpha-4,\epsilon,s}(\tilde{H})\\
&\lesssim_s&\|[(\Delta_{g,f}+1)\tilde{h}]e^{F_{\alpha-1,\epsilon,s}}\|^2_{L^2_f}+\epsilon I^0_{\alpha,\epsilon,s}(\tilde{h})+C(K_{\alpha}).
\end{eqnarray*}
Therefore,
\begin{eqnarray*}
(\alpha-\epsilon)I^0_{\alpha,\epsilon,s}(\tilde{h})\lesssim_s\|[(\Delta_{g,f}+1)\tilde{h}]e^{F_{\alpha,\epsilon,s}}\|^2_{L^2_f}\lesssim_s \epsilon I^0_{\alpha,\epsilon,s}(\tilde{h})+C(K_{\alpha}),
\end{eqnarray*}
which implies, for $\epsilon$ small enough,
\begin{eqnarray*}
I^0_{\alpha,\epsilon,s}(\tilde{h})\leq C(\alpha,\epsilon),
\end{eqnarray*}
where $C(\alpha,\epsilon)$ is a positive constant independent of $s$. Hence $\alpha+\epsilon\in S$ by reasoning as above. In particular, $\alpha_++\epsilon\in S$ for some small $\epsilon$. Contradiction.

\end{proof}

\end{proof}

We end this section by proving lemma \ref{com-ric-cur-inf}.

\begin{proof}[Proof of lemma \ref{com-ric-cur-inf}]
\begin{itemize}
\item Let $(\Sigma^2,g,\nabla^gf)$ be an asymptotically conical expanding gradient Ricci soliton. 
Then, 
\begin{eqnarray}\label{const-resc-curv-dim-2}
\nabla^g(e^f\R_g)&=&(\R_g\nabla^gf+\nabla^g \R_g)e^f=(\R_g\nabla^gf-2\Ric(g)(\nabla^gf))e^f=0,
\end{eqnarray}
by the soliton identity (\ref{equ:2}) and the fact that $\Ric(g)=(\R_g/2)g$ in dimension $2$. As $f$ is proper, it attains its minimum at some point $p\in \Sigma$, i.e. by (\ref{equ:3}) and (\ref{const-resc-curv-dim-2}),
\begin{eqnarray*}
\Rinf(g)=\lim_{+\infty}e^{f+\mu(g)}\R_g=e^{f(p)+\mu(g)}\R_g(p)=e^{\min_{\Sigma}f+\mu(g)}(\min_{\Sigma}f+\mu(g)).
\end{eqnarray*}

\item The computation of $\Rinf(g)$ in the Feldman-Ilmanen-Knopf examples is mainly due to Siepmann [Remark $3.3.4$, \cite{Sie-PHD}]. One can also give a complete description of $\Ricinf(g)$ in this example.
\end{itemize}
\end{proof}

\appendix

\section{Soliton equations}\label{sol-equ-sec}

The next lemma gathers well-known Ricci soliton identities together with the (static) evolution equations satisfied by the curvature tensor.

Recall first that an expanding gradient Ricci soliton is said \textit{normalized} if $\int_Me^{-f}d\mu_g=(4\pi)^{n/2}$ (whenever it makes sense).

\begin{lemma}\label{id-EGS}
Let $(M^n,g,\nabla^g f)$ be a normalized expanding gradient Ricci soliton. Then the trace and first order soliton identities are :
\begin{eqnarray}
&&\Delta_g f = \R_g+\frac{n}{2}, \label{equ:1} \\
&&\nabla^g \R_g+ 2\Ric(g)(\nabla^g f)=0, \label{equ:2} \\
&&\arrowvert \nabla^g f \arrowvert^2+\R_g=f+\mu(g), \label{equ:3}\\
&&\div_g\Rm(g)(Y,Z,T)=\Rm(g)(Y,Z, \nabla f,T),\label{equ:4}
\end{eqnarray}
for any vector fields $Y$, $Z$, $T$ and where $\mu(g)$ is a constant called the entropy.\\

The evolution equations for the curvature operator, the Ricci tensor and the scalar curvature are :
\begin{eqnarray}
&& \Delta_f \Rm(g)+\Rm(g)+\Rm(g)\ast\Rm(g)=0,\label{equ:5}\\
&&\Delta_f\Ric(g)+\Ric(g)+2\Rm(g)\ast\Ric(g)=0,\label{equ:6}\\
&&\Delta_f\R_g+\R_g+2\arrowvert\Ric(g)\arrowvert^2=0,\label{equ:7}
\end{eqnarray}
where, if $A$ and $B$ are two tensors, $A\ast B$ denotes any linear combination of contractions of the tensorial product of $A$ and $B$.
\end{lemma}

\begin{proof}
See [Chap.$1$,\cite{Cho-Lu-Ni-I}] for instance.

\end{proof}

\begin{prop}\label{pot-fct-est}
Let $(M^n,g,\nabla^g f)$ be an expanding gradient Ricci soliton.
\begin{itemize}
\item If $(M^n,g,\nabla^g f)$ is non Einstein, if $v:=f+\mu(g)+n/2$,
\begin{eqnarray}
&&\Delta_fv=v,\quad v>\arrowvert\nabla v\arrowvert^2.\label{inequ:1}\\
\end{eqnarray}

\item Assume $\Ric(g)\geq 0$ and assume $(M^n,g,\nabla^g f)$ is normalized. Then $M^n$ is diffeomorphic to $\mathbb{R}^n$ and
\begin{eqnarray}
&&v\geq \frac{n}{2}>0.\label{inequ:2}\\
&&\frac{1}{4}r_p(x)^2+\min_{M}v\leq v(x)\leq\left(\frac{1}{2}r_p(x)+\sqrt{\min_{M}v}\right)^2,\quad \forall x\in M,\label{inequ:3}\\
&&\AVR(g):=\lim_{r\rightarrow+\infty}\frac{\vol B(q,r)}{r^n}>0,\quad\forall q\in M,\label{inequ:avr}\\
&&-C(n,V_0,R_0)\leq\min_{M}f\leq 0\quad;\quad\mu(g)\geq\max_{M}\R_g\geq 0,\label{inequ:ent}
\end{eqnarray}
where $V_0$ is a positive number such that $\AVR(g)\geq V_0$, $R_0$ is such that $\sup_{M}\R_g\leq R_0$ and $p\in M$ is the unique critical point of $v$.\\

\item Assume $\Ric(g)=\textit{O}(r_p^{-2})$ where $r_p$ denotes the distance function to a fixed point $p\in M$. Then the potential function is equivalent to $r_p^2/4$ (up to order $2$).
\end{itemize}
\end{prop}

For a proof, see \cite{Der-Asy-Com-Egs} and the references therein.

\bibliographystyle{alpha.bst}
\bibliography{bib-uni-con-ric-sol}

\end{document}